\documentclass[12pt]{extarticle}
\usepackage{amsfonts,amsmath}
\usepackage{amssymb}
\usepackage{latexsym}
\usepackage{epsfig}
\usepackage{graphicx,color,graphics}
\usepackage{caption}
\usepackage{subcaption}
\usepackage{xcolor}
\usepackage{tcolorbox,color}
\usepackage{lineno}
\usepackage{pgf,tikz}
\usepackage{mathrsfs}
\usepackage{multicol}
\usepackage{enumitem}
\usepackage{mdwlist}
\usepackage{dsfont}
\usetikzlibrary{arrows}
\usetikzlibrary{patterns}
\newtheorem{lemma}{Lemma}[section]

\newtheorem{theorem}[lemma]{Theorem}
\newtheorem{proposition}[lemma]{Proposition}
\newtheorem{remark}[lemma]{Remark}

\newcommand{\N}{\ifmmode{{\Bbb N}}\else{\mbox{${\Bbb N}$}}\fi}
\newcommand{\R}{\ifmmode{{\Bbb R}}\else{\mbox{${\Bbb R}$}}\fi}

\providecommand{\inte}[2]{\overset{#1}{\underset{#2}{\int}}}

\begin{document}
		\title{On the regularity for thermoelastic systems of phase-lag parabolic type}
\author{Jaime Mu\~noz Rivera$^{1,}$\footnote{jemunozrivera@gmail.com}, \ Elena Ochoa Ochoa $^{2,}$ \footnote{elenaochoaochoa18@gmail.com}, \ \  Ramón Quintanilla $^{3,}$  \footnote{ramon.quintanilla@upc.edu}.  \\   \small \it  $^{1}$Universidad del B\'io B\'io, Departamento de Matem\'aticas, Concepci\'on, Chile.\\
	\small \it $^{1}$ Laboratório Nacional de Computação Científica, Petrópolis, Brasil.\\
	 \small \it $^{2}$ Universidad Andres Bello, Departamento de Matem\'aticas, Facultad de Ciencias Exactas,\\
	\small \it Sede Concepci\'on, Autopista Concepci\'on-Talcahuano 7100, Talcahuano, Chile.\\
		\small \it $^{3}$ Universidad Polit\'ecnica de la Catalunya,  Departamento de Matem\'aticas,\\
	\small \it Colom 11,08222 Terrassa, Espa\~na.			    
}		

\date{}
		\maketitle

\begin{abstract}

In this article, we investigate the maximal smoothness (infinite differentiability) of solutions to thermoelastic models, specifically those where the heat equation is of the ``phase-lag'' or ``parabolic'' type. We derive optimal regularity results for two distinct models. The first model addresses the transverse oscillations of a fully thermoelastic plate, for which we prove that the associated semigroup is analytic. The second model considers a partially thermoelastic plate composed of two components: a thermoelastic component with nonzero temperature differences and an elastic component unaffected by temperature variations. For this model, we demonstrate that the semigroup \( S(t) \) belongs to the Gevrey class of order 4, provided the solutions are radial and symmetric. Both analyticity and Gevrey class membership are qualitative properties that intricately link regularity and stability, driven by robust dissipative mechanisms. These properties are significantly stronger than standard regularity conditions, such as belonging to the class \( C^k \) or a Sobolev space \( H^s \).

\end{abstract}


		\section{Introduction}
			It is well known that the juxtaposition of Fourier's law with the energy equation leads naturally to the paradox of the instantaneous propagation of thermal waves. It is also accepted that the propagation of heat at low temperatures is not well described by this form of the heat equation \cite{HETm,HETm1}. For this reason, Cattaneo and Maxwell \cite{Cattaneo} proposed an alternative theory to describe the evolution of temperature. Since the 1970s, a large number of alternative theories for heat conduction have been proposed. With each of these theories can be associated with a thermoelastic theory. Today there are a large number of thermoelastic theories and contributions to their study. It is also worth recalling the recent contributions of Iesan involving high order spatial derivatives for the heat conduction \cite{I1,I2,I3,I4,I5}.
		
			A couple of these alternative theories are those proposed by Tzou \cite{tzou} and Choudhuri \cite{Roy} who provided a constitutive relationship for the heat flow by introducing delay parameters  (usually called "phase-lag"). However, these proposals are not adequate either as they lead to a strongly explosive behavior that does not correspond to what is obtained empirically when studying heat \cite{MQR}. However, the replacement of these functions by different approximations using Taylor polynomials has been accepted by the scientific community and the number of contributions based on this type of theory is immense.
		
		In 2018, Magaña and Quintanilla \cite{MaQ1} studied the problem determined by the elasticity system coupled with a heat equation of the type discussed above. They obtained that the solutions to this problem can be determined by a "quasi-contractive" semigroup. This fact allows to conclude the existence, the uniqueness and the continuous dependence of the solutions with respect to the initial data and the supply terms. However, stability properties have not been obtained in general, and indeed it seems that no conclusion  can be drawn unless we can restrict ourselves to some sub-class of problems. At the same time, the regularity of the solutions has not been studied in detail.
		
				This paper considers this last aspect. We analyze a thermoelastic systems when the heat equation is of the "phase-lag" of "parabolic" type. Our objective is to show results of regularity of the solutions. We will consider two cases, when the material is thermoelastic  type acting on the whole domain and when the material has a localized thermoelastic component, that is, when the material has two components, one simply elastic, without dissipative mechanisms and the other component a thermoelastic systems of phase-lag parabolic type.

\begin{itemize}
\item Case 1: In section 2, let  $\Omega\subset \mathbb{R}^{2}$ be  a bounded domain  with smooth boundary $\partial \Omega$. The fully thermoelastic system of phase-lag parabolic type is given by 
\begin{align}
	\rho u_{tt}&=-\kappa \Delta ^2 u-\beta \Delta \theta \quad \text{in} \ \Omega \times \mathbb{R}^{+}_{0}, \label{b1}\\
	c_T\dfrac{\partial}{\partial t}\left(\sum_{j=0}^na_j\theta^{(j)}\right)&= \left(\sum_{j=0}^nb_j\Delta \theta^{(j)}\right) +\beta \dfrac{\partial}{\partial t}\left(\sum_{j=0}^na_j\Delta u^{(j)}\right)\quad \quad \text{in} \ \Omega \times \mathbb{R}^{+}_{0},\label{b2}
\end{align}	
verifying  the boundary conditions 
\begin{align}\label{frontera}
	u=\Delta u=\theta=0  \ \ \text{on} \  \ \partial \Omega,  \quad t>0
\end{align}
and  initial conditions
\begin{align}
	u(x,0)=u_0(x),\;\; u_t(x,0)=u_1(x), \quad \theta^{(j)}(x,0)=\eta_j(x),\quad  j=0,\cdots, n. \label{b33n2p}
\end{align}

Here $u$ is the displacement, $\theta$ is the temperature, $\rho$ is the mass density,  $c_T$  is the heat capacity $k$ is the elasticity constant, $a_n$ is the phase-lag parameters for heat flux, $b_n$  is the phase-lag parameters temperature gradient and $\beta$ is thermoelastic coupling coefficient.  We will prove that the semigroup is associated with the system \eqref{b1}-\eqref{b33n2p} is analytic.

\item Case 2:  In Section \ref{Gevrey}, we analyze a partially thermoelastic model defined on a composite plate. To this end, we define the domain configuration as follows. Let \( \Omega_0 \) and \( \Omega_2 \subset \mathbb{R}^2 \) be bounded domains with smooth boundaries, satisfying \( \overline{\Omega}_2 \subset \Omega_0 \). Define \( \Omega_1 = \Omega_0 \setminus \overline{\Omega}_2 \), and let the plate's domain be \( \Omega = \Omega_1 \cup \Omega_2 \). The exterior component \( \Omega_1 \) consists of the thermoelastic material, while the interior component \( \Omega_2 \) comprises the elastic material. The boundary of \( \Omega \) is denoted by \( \Gamma_0 \), and the interface between \( \Omega_1 \) and \( \Omega_2 \) is denoted by \( \Gamma_1 \).

\begin{figure}[h]\label{Figur1}
\setlength{\unitlength}{2.0pt}
\begin{center}
\begin{picture}(100,75)

\put(7,-7){\includegraphics[width=6.2cm]{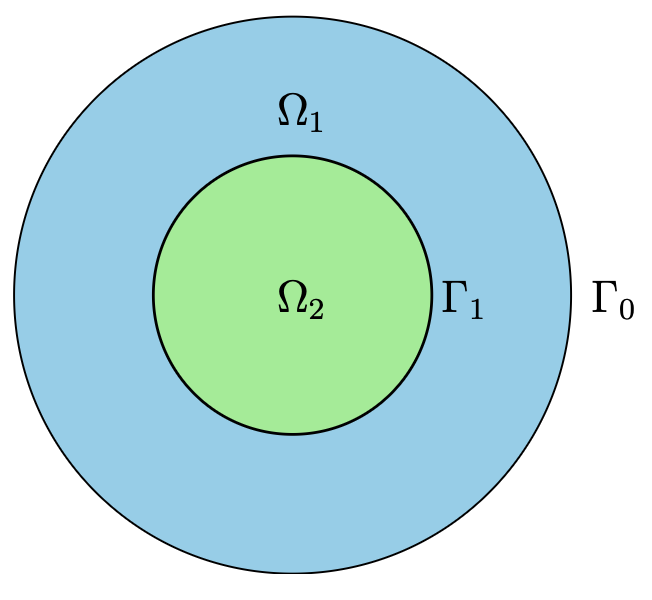}}

\end{picture}
\end{center}
	\caption{Localized thermal effect on $\Omega_1$}
\end{figure}

To facilitate notations let us  denote by $u$ the transversal oscillation of the termoelastic component over $\Omega_1$ and let us denote by $v$ the transversal oscillation of the elastic component over $\Omega_2$
\begin{align}
	\rho u_{tt}&=-\kappa_1 \Delta ^2 u-\beta \Delta \theta  &\text{in} \ \Omega_1 \times \mathbb{R}^{+}_{0}, \label{bl1}\\
	c_T\dfrac{\partial}{\partial t}\left(\sum_{j=0}^na_j\theta^{(j)}\right)&= \left(\sum_{j=0}^nb_j\Delta \theta^{(j)}\right) +\beta \dfrac{\partial}{\partial t}\left(\sum_{j=0}^na_j\Delta u^{(j)}\right)\quad  &\text{in} \ \Omega_1 \times \mathbb{R}^{+}_{0},\label{bl2}\\
	\rho v_{tt}&=-\kappa_2 \Delta ^2 v&\text{in} \ \Omega_2 \times \mathbb{R}^{+}_{0}, \label{bl3}
\end{align}	
verifying  the boundary conditions 
\begin{align}\label{fronteraloc}
	u=\Delta u=\theta=0  \ \ \text{on} \  \ \Gamma_0, \quad  {\theta}=0 \ \ \text{on} \  \  \Gamma_1,  \quad t>0.
\end{align}
Verifying the transmission conditions 
\begin{equation}\label{trans1}
	{u}(x,t)={v}(x,t),\quad \frac{\partial {u}}{\partial\nu}(x,t)=\frac{\partial {v}}{\partial\nu}(x,t) \ \ \text{on} \  \ \Gamma_1,
\end{equation}
\begin{equation}\label{trans2}
	\kappa_1 \Delta {u}+\beta{\theta} =\kappa_2 \Delta{v} ,\quad \kappa_1\frac{\partial \Delta {u}}{\partial\nu}+\beta \frac{\partial {\theta}}{\partial\nu}=\kappa_1\frac{\partial \Delta {v}}{\partial\nu} \ \ \text{on} \  \ \Gamma_1.
\end{equation}
Additionally, we consider the following initial conditions
\begin{align}
	u(x,0)=u_0(x),\;\; u_t(x,0)=u_1(x), \quad \theta^{(j)}(x,0)=\eta_j(x),\quad  j=0,\cdots, n. \label{b33n2}
\end{align}

 We will prove that the semigroup is associated with the system \eqref{bl1}-\eqref{b33n2} is Gevrey class 4,  provided that the solutions are radial and symmetrical.
\end{itemize}

The primary contribution of this work is to prove that the semigroup associated with the thermoelastic model is analytic when thermal effects are uniform across the entire domain. We then investigate the influence of localized thermal dissipation on the model’s dynamics. In this case, we establish that the corresponding semigroup belongs to the Gevrey class of order 4, provided the solutions are radial and symmetric. 
This  result is new and  highlights the regularizing effect of localized thermal mechanisms, demonstrating that such dissipation induces maximal smoothness (infinite differentiability) in the solutions. In particular our result implies

\begin{itemize}
\item If  the semigroup belongs to the Gevrey class of order $\mu$ ($0<\mu<1$) then it is instantaneous differentiable (which implies the maximal smoothness) and  verifies 
$$
\limsup_{t\rightarrow 0 }t^{\frac 2\mu-1}\|Ae^{At}\|<\infty.
$$
See \cite{Crand-Pazy} 
\item From the above property we conclude that the semigroup  is 
immediately norm continuous (immediately uniformly continuous), see \cite{Engel}. 
\item Finally, the norm continued property implies  the spectrum determined growth property (SDG), that means that the growth abscissa (type) of the semigroup 
$$
\omega_0=\lim_{t\rightarrow\infty}\frac{\ln(\|e^{At}\|)}{t}
$$
is equals to upper bound of the spectrum of $A$. This property is important to find numerically the growth abscissa. 
\end{itemize}

The remaining part of this paper is organized as follows. In Section 2, we establish the well-posedness of the model and prove that the associated semigroup, which governs the solutions, is analytic. In Section 3, we show that the corresponding semigroup belongs to the Gevrey class of order 4, provided the solutions are radial and symmetric.

\section{Global heat conduction}
Denoting by 
	$f^{(k)}=\dfrac{\partial^k f}{\partial t^k}$ 
let us  introduce the operator 
	\begin{align*}
		\widehat{f}=a_0 f+a_1f '+...+a_n f^{(n)}=\sum_{j=0}^na_jf^{(j)}.
	\end{align*}
Differentiating $n$ times system \eqref{b1}-\eqref{b2} multiplying any $(j)$ derivative by $a_j$ and summing up the product result we arrive to 

\begin{eqnarray}
	\rho \widehat{u}_{tt}&=&-\kappa\Delta^2 \widehat{u}-\beta \Delta \widehat{\theta},\label{ep1}\\
	c_T\dfrac{\partial}{\partial t}\widehat{\theta}&=&\sum_{j=0}^nb_j\Delta \theta^{(j)}+\beta \Delta \widehat{u_t}. \label{ep2}
\end{eqnarray}
Multiplying equation \eqref{ep1}  by $\widehat{u}_{t}$ and equation \eqref{ep2} by $\widehat{\theta}$ we get
\begin{align*}
	\frac 12 \frac{d}{dt}\left(\int_{\Omega }  \rho\left|\widehat{u}_{t} \right| ^2+ \kappa\left| \Delta \widehat{u} \right| ^2\ d\Omega \right)-\beta\int_{\Omega } \nabla \widehat{\theta}\cdot\nabla \widehat{u}_t \;d\Omega=0,
\end{align*}
\begin{align*}
	\frac 12 \frac{d}{dt}\left(\int_{\Omega }  c_T|\widehat{\theta} | ^2 d\Omega \right)+\int_{\Omega } \sum_{j=0}^nb_j\nabla\theta^{(j)}\cdot\nabla \widehat{\theta}\; d\Omega+\beta\int_{\Omega } \nabla \widehat{\theta}\cdot\nabla \widehat{u}_t\; d\Omega=0.
\end{align*}
Summing the above identities we get 
\begin{align*}
	\frac 12 \frac{d}{dt}\left(\int_{\Omega }  \rho\left|\widehat{u}_{t} \right| ^2+ \kappa\left| \Delta \widehat{u} \right| ^2 + c_T| \widehat{\theta} | ^2 \;d\Omega \right)=-\sum_{j=0}^n\int_{\Omega } b_j\nabla\theta^{(j)}\cdot\nabla \widehat{\theta}\; d\Omega.
\end{align*}
 On the other hand, differentiating the expression 
$$
\frac 12 \frac{d}{dt}\int_{\Omega }  \sum_{j=0}^{n-1}| \nabla\theta^{(j)}|^2\;d\Omega=
\int_{\Omega }  \sum_{j=0}^{n-1} \nabla\theta^{(j+1)}\cdot \nabla\theta^{(j)}\;d\Omega.
$$
Denoting by $E$ the energy defined as 
\begin{align*}
 E(t)=\frac 12 \int_{\Omega }  \rho\left|\widehat{u}_{t} \right| ^2+ \kappa\left| \Delta \widehat{u} \right| ^2 + c_T| \widehat{\theta} | ^2+ \sum_{j=0}^{n-1}| \nabla\theta^{(j)}|^2 \;d\Omega.
\end{align*}
We have that 
\begin{align}
	\dfrac{d}{dt}E(t)=-\sum_{j=0}^n\int_{\Omega } b_j\nabla\theta^{(j)}\cdot\nabla \widehat{\theta}\; d\Omega+\int_{\Omega }  \sum_{j=0}^{n-1} \nabla\theta^{(j+1)}\cdot \nabla\theta^{(j)}\;d\Omega\label{Energ}.
\end{align}
Note that 
\begin{eqnarray}
-\sum_{j=0}^n\int_{\Omega } b_j\nabla\theta^{(j)}\cdot\nabla \widehat{\theta}\; d\Omega&=&-\sum_{j=0,i=0}^n\int_{\Omega } b_j\nabla\theta^{(j)}\cdot\left(\sum_{j=0,i=0}^na_i\nabla \theta^{(i)}\right) d\Omega\nonumber\\
&=&-\sum_{j=0,i=0}^n\int_{\Omega } a_ib_j\nabla\theta^{(j)}\cdot\nabla \theta^{(i)}d\Omega.\label{PreL}
\end{eqnarray}
Since  $a_n$ and $b_n$ are positive numbers, we get
\begin{eqnarray*}
\sum_{i,j=0}^na_ib_jY_iY_j&=&\left( \sum_{i=0}^{n-1}a_iY_i +a_nY_n \right) \left( \sum_{j=0}^{n-1}b_jY_j +b_nY_n \right) \\
&=&\sum_{i,j=0}^{n-1}a_ib_jY_iY_j+a_nY_n\sum_{i=0}^{n-1}b_jY_j+b_nY_n\sum_{i=0}^{n-1}a_iY_i+a_nb_nY_n^2\\
&\geq & -c\sum_{i=0}^{n-1}Y_i^2+\frac{a_nb_n}{2}Y_n^2.
\end{eqnarray*}
Using the above inequality into \eqref{PreL} we conclude from \eqref{Energ} that 
\begin{align}
	\dfrac{d}{dt}E(t)\leq -\frac{a_nb_n}{2}\int_{\Omega }|\nabla\theta^{(n)}|^2\; d\Omega+c_0\int_{\Omega }  \sum_{j=0}^{n-1}| \nabla\theta^{(j)}|^2\;d\Omega
	\label{EnergIn}.
\end{align}
for  $c_0$ a positive constant.\\

Note that the system is not dissipative.

\subsection{Existence: Semigroup approach}

From now on and without lost of generality we assume that $\rho=c_T=1$. The phase space we consider is given by 
\begin{align*}
	\mathcal{H}=[H_0^1(\Omega)\cap H^2(\Omega)]\times L^2(\Omega) \times \left[H_0^1(\Omega) \right]^n\times L^2 (\Omega),
\end{align*}
where $H_0^1(\Omega), H^2(\Omega)$ and $L^2(\Omega)$ are the well known Sobolev spaces. 
From now on we will use $\mathbf{u},\mathbf{v},\vartheta$ instead of $\widehat{u},\widehat{v},\widehat{\theta}$. For any two elements of $	\mathcal{H}$  
\begin{align*}
	U=(\mathbf{u},\mathbf{v},\Theta_0,\Theta_1,\cdots,\Theta_n),\quad U^*=(\mathbf{u}^*,\mathbf{v}^*,\Theta_0^*,\Theta_1^*,\cdots,\Theta_n^*).
\end{align*}
where $\Theta_j=\theta^{(j)}$ for $j=0,\cdots n$.  Denoting 
\begin{align}
	\vartheta= a_0\Theta_0+a_1\Theta_1+...+a_n\Theta_n \nonumber.
\end{align}
The inner product we consider to $\mathcal{H}$ is 
$$
(U,U^*)_{\mathcal{H}}=\int_{\Omega}   \mathbf{v}  \overline{\mathbf{v}}^* 
+\kappa \Delta \mathbf{u}\,\overline{\Delta \mathbf{u}}^*  +\sum_{j=0}^{n-1}\nabla \Theta_{j}  \overline{\nabla \Theta_{j}}^*+  \left( \sum_{j=0}^{n}a_j\Theta_j \right) \left( \sum_{j=0}^{n}a_j \overline{\Theta_j ^*}\right)    d\Omega,
$$
hence 
\begin{eqnarray*}
\left\| 	U\right\| _{\mathcal{H}}^{2}&=&\int_{\Omega}\left( \left| \mathbf{v}\right|^2
+\kappa \left|\Delta \mathbf{u} \right| ^2 +\sum_{j=0}^{n-1}\left|\nabla \Theta_{j} \right|^2+\left| \sum_{j=0}^{n}a_j\Theta_j \right|^2    \right)  d\Omega,\\
&=&\int_{\Omega}\left( \left| \mathbf{v}\right|^2
+\kappa \left|\Delta \mathbf{u} \right| ^2 +\sum_{j=0}^{n-1}\left|\nabla \Theta_{j} \right|^2+\left| \vartheta  \right|^2    \right)  d\Omega.
\end{eqnarray*}
Under these conditions problem \eqref{b1}-\eqref{b2} can be rewritten as 
\begin{align*}
	\dfrac{dU}{dt}=\mathcal{A}U, \ \  U(0)=U_0\in \mathcal{H},
\end{align*}
where 

	\begin{align}
	\mathcal{A}U=\begin{pmatrix}
	\mathbf{v}\\
		-\Delta(\kappa\Delta \mathbf{u} +\beta(a_0\Theta_0+a_1\Theta_1+...+a_n\Theta_n))\\
		\Theta_1 \\
		\Theta_2\\
		.\\
		.\\
		.\\
		\Theta_n\\
	\displaystyle	\dfrac{\beta}{a_n}\Delta \mathbf{v}+\dfrac{1}{a_n} \sum_{j=0}^{n}b_j\Delta\Theta_j-\dfrac{1}{a_n}\sum_{j=0}^{n-1}a_j\Theta_{j}
	\end{pmatrix} . 
\label{opeA}
\end{align}
The domain of the operator is  
\begin{align*}
	D(\mathcal{A})=\left\{U\in\mathcal{H};\; \mathbf{v}\in H_0^{1}\cap H^{2},  \ \sum_{j=0}^n b_j\Theta_j\in H^2(\Omega), \;\;\kappa\Delta 	\mathbf{u}+\beta	\vartheta\in H^{2}(\Omega), \mbox{verifying }\eqref{frontera}\right\} .
\end{align*}
Note that 
\begin{align}
\left\langle \mathcal{A}U,U \right\rangle &=	 \int_{\Omega}  \kappa\Delta \mathbf{v}\overline{\Delta \mathbf{u}}-\kappa\Delta \mathbf{u}\overline{\Delta \mathbf{v}}+\beta \nabla\vartheta \overline{\nabla \mathbf{v}}-\beta \nabla\mathbf{v}\overline{\nabla \vartheta }\nonumber\\
&+\sum_{k=0}^{n-1}\nabla \Theta_{k+1}\overline{\nabla\Theta}_k-\sum_{k,\ell=0}^{n}a_k b_\ell\nabla \Theta_{\ell}\overline{\nabla\Theta}_{k} d\Omega,\label {Operator}
\end{align}
Taking the real part we have

\begin{align*}
	\mbox{Re }\left\langle \mathcal{A}U,U \right\rangle =	\mbox{Re } \int_{\Omega}\left( \sum_{k=0}^{n-1}\nabla \Theta_{k+1}\overline{\nabla\Theta}_k-\sum_{k,\ell=0}^{n}a_k b_\ell\nabla \Theta_{\ell}\overline{\nabla\Theta}_{k}\right) \ d\Omega.
\end{align*}
Using similar reasoning as in \eqref{EnergIn} . 
\begin{align*}
	\mbox{Re }\left\langle \mathcal{A}U,U \right\rangle\leq -\frac{a_nb_n}{2}\int_{\Omega }|\nabla\Theta_{n}|^2\; d\Omega+c_0\int_{\Omega }  \sum_{j=0}^{n-1}| \nabla\Theta_{j}|^2\;d\Omega,
\end{align*}
which is not dissipative. 
To apply the semigroup theory of dissipative generators, we consider a continuous perturbation of $\mathcal{A}$ given by $\mathfrak{B}=\mathcal{A}-2c_0\mathbf{I}$. It is clear that 
$D(\mathfrak{B})=D(\mathcal{A})$ and 
\begin{align}\label{nondis}
	\mbox{Re }\left\langle \mathfrak{B}U,U \right\rangle\leq -\frac{a_nb_n}{2}\int_{\Omega }|\nabla\Theta_{n}|^2\; d\Omega-c_0\left\|U \right\|^2_{\mathcal{H}} .
\end{align}
 Hence to show that $\mathfrak{B}$ is the infinitesimal generator of a contraction semigroup, it is enough to show that $0\in \varrho(\mathfrak{B})$. Indeed, denoting by $\mu=2c_0$,  we have to show that for any $F\in \mathcal{H}$, there exists only one $U\in D(\mathfrak{B})$ such that 
$-\mathfrak{B}U=F$. That is to say 
\begin{align}\label{RES}
	\mu U-\mathcal{A}U=F.
\end{align}
In terms of the components the resolvent equation is  written as 
\begin{align}
	\mu \mathbf{u}-\mathbf{v}=f_1,\label{res1}\\
	\mu \mathbf{v}+ \kappa\Delta ^2\mathbf{u}+\beta\Delta\vartheta =f_2,\label{res2}\\
	\mu\Theta_k-\Theta_{k+1}=g_k,\label{res3}\\
	\mu \Theta_n-\dfrac{1}{a_n}(\beta\Delta \mathbf{v}+\sum_{j=0}^{n}b_j\Delta\Theta_{j})-\dfrac{1}{a_n}\sum_{j=0}^{n-1}a_j\Theta_{j}=g_n,\label{resn}
\end{align}
where  $k=0,\;,\cdots \; n-1$. Under the above conditions we have 
\begin{theorem}\label{Existe1}
	The operator $\mathcal{A}$ generates a quasi-contractive semigroup.
\end{theorem}

\begin{proof} We show that $\mathfrak{B}=\mathcal{A}-2c_0\mathbf{I}$ is the infinitesimal generator of a contraction semigroup that in particular implies that  $\mathcal{A}$ generates a quasi-contractive semigroup.
In fact, to do that it is enough to show that  $0\in \varrho(\mathfrak{B})$. Denoting by 
$$
U^{i}=(\mathbf{u}^i, \mathbf{v}^i, \Theta_{0}^i, \cdots , \Theta_{n}^i),\quad i=1,2. 
$$
Let us denote by 
$$
a(U^1,U^2)=(-\mathcal{A}U^1,U^2)_{\mathcal{H}}+2c_0(U^1,U^2)_{\mathcal{H}}
$$
From identity \eqref{Operator}  we conclude that  $a(\cdot,\cdot)$ is a continue bilinear form over $\mathcal{H}$. Moreover 
using inequality \eqref{nondis} we get that the bilinear form is  coercive. 
$$
a(U,U)\geq \frac{a_nb_n}{2}\int_{\Omega }|\nabla\Theta_{n}|^2\; d\Omega+c_0\left\|U \right\|^2_{\mathcal{H}} .
$$
Using the Lax-Milgram Lemma we conclude that for any $F\in \mathcal{H}$ there exists only one solution $U\in\mathcal{H}$ such that 
$$
a(U,U^2)=(F,U^2)_{\mathcal{H}},\quad\forall U^2\in\mathcal{H}
$$
but from the definition of  $a(\cdot,\cdot)$ this implies that 
$$
-\mathfrak{B}U=F, 
$$
in the distributional sense. Then using the elliptic regularity we conclude that $U\in D(\mathcal{A})$. 
Therefore   $\mathfrak{B}$ is the infinitesimal generator of a semigroup of contractions.

\end{proof}

\bigskip

\subsection{Analyticity}
Here we prove that the operator $\mathfrak{B}$ given in \eqref{opeA} generates  an analytic  semigroup wherever  $\beta$ is different from zero. To do that  we use the following  characterization that we can find at  Liu and Zheng \cite{liub} to Hilbert spaces. See also Pazy \cite{Pazy} and Klaus-Jochen E. and Rainer N. \cite{Engel}.

\begin{theorem} \label{TeoLiu} A contractive semigroup $S(t)=e^{\mathcal{A}t}$ is analytic over a Hilbert space $\mathcal{H}$ if and only if 
$i\mathbb{R}\subset \varrho({\mathcal{A}})$ and 
$$
\|\gamma(i\gamma I-\mathcal{A})^{-1}\|_{\mathcal{L}(\mathcal{H})}\leq C,\quad \forall \gamma\in \mathbb{R}.
$$
\end{theorem}

\noindent 
\begin{proof}
See  \cite{liub}.
\end{proof}

\bigskip 
\noindent
To show the analyticity of $S(t)=e^{\mathcal{A}t}$ we consider the continuous perturbation of $\mathcal{A}$ given by
$\mathfrak{B}=\mathcal{A}-2c_0\mathbf{I}$. Since the identity commutes with  $\mathcal{A}$ then we have that 
$$
S(t)=e^{\mathcal{A}t}=e^{\mathfrak{B}t}e^{2c_0t}.
$$
Therefore $e^{\mathcal{A}t}$ is analytic if and only if $e^{\mathfrak{B}t}$ is analytic. Since $\mathfrak{B}$ is dissipative \eqref{nondis},  we can apply Theorem \ref{TeoLiu} to show that $e^{\mathfrak{B}t}$ is analytic. To do that we consider the resolvent equation: 
$$
i\gamma U-\mathfrak{B}U=F. 
$$
Taking the inner product in $\mathcal{H}$ with $U$ and taking the real part we get
\begin{align*} 
\int_{\Omega} \left| \mathbf{v}\right|^2
+\kappa \left|\Delta\mathbf{ u} \right| ^2 +\sum_{j=0}^{n-1}|\nabla \Theta_{j} |^2+c|\vartheta |^2    d\Omega+\frac{a_nb_n}{2}\int_{\Omega }|\nabla\Theta_{n}|^2\; d\Omega\leq c\|U\|_{\mathcal{H}}\|F\|_{\mathcal{H}},
\end{align*}
or 
\begin{align}\label{Dissipa}
\int_{\Omega} \left| \mathbf{v}\right|^2
+\kappa \left|\Delta\mathbf{ u} \right| ^2 +\sum_{j=0}^{n}|\nabla \Theta_{j} |^2+|\vartheta |^2    d\Omega\leq c\|U\|_{\mathcal{H}}\|F\|_{\mathcal{H}}.
\end{align}

For some positive constant $c$. Hence 
system \eqref{res1}-\eqref{resn} can be written as 
\begin{align}
	i\gamma \mathbf{u}+\mu \mathbf{u}-\mathbf{v}=f_1\label{xres1},\\
	i\gamma \mathbf{v}+\mu \mathbf{v}+  \kappa\Delta ^2\mathbf{u}-\beta\Delta \vartheta  =f_2,\label{xres2}\\
	i\gamma \Theta_k +\mu \Theta_k-\Theta_{k+1}=g_k,\label{xres3}\\
	i\gamma \Theta_n+\mu \Theta_n-\dfrac{1}{a_n}\left( \beta\Delta \mathbf{v}+\sum_{j=0}^{n}b_j\Delta\Theta_{j}\right) +\dfrac{1}{a_n}\beta\sum_{j=0}^{n-1}a_j\Theta_{j}=g_n,\label{xresn}
\end{align}
where  $k=0,\;,\cdots \; n-1$.
Our first step is to show that the imaginary axes is contained in the resolvent set of the operator $\mathfrak{B}$. To do that we use the following result.
\begin{lemma} \label{iRcR}
 Under the above conditions we have that $i\mathbb{R}\subset \varrho(\mathfrak{B})$. 
 \end{lemma}
\begin{proof}
			Let us denote by 
\begin{equation*}
\mathcal{N}=\{s\in \mathbb{R}^+\colon ]-is,\,is[\subset \varrho(\mathfrak{B})\}.
\end{equation*}
Since $0\in \varrho(\mathfrak{B})$ so we have  $\mathcal{N}\neq \varnothing$. Putting  $\sigma=\sup \mathcal{N}$ we have to possibilities:  first  $\sigma =+\infty$ which implies that $i\mathbb{R}\subseteq \varrho(\mathfrak{B})$, and  that $0<\sigma$ finite. We will reason by contradiction. Let us suppose  that  $\sigma<\infty$. Then, exists a sequence $\{\gamma_n\}\subseteq \mathbb{R}$ such that $ \gamma_n\to\sigma< \infty$ and
$$\|(i\gamma_nI-\mathfrak{B})^{-1}\|_{\mathcal{L}(\mathcal{H})}\to \infty.$$
Hence, there exists a sequence $\{f_n\}\subseteq \mathcal{H}$ verifying  $\|f_n\|_{\mathcal{H}}=1$ and $\|(i\gamma_nI-\mathfrak{B})^{-1}f_n\|_{\mathcal{H}}\to \infty$. Denoting by
$$
\tilde{U}_n=(i\gamma_nI-\mathfrak{B})^{-1}f_n\quad \Rightarrow \quad f_n=i\gamma_n\tilde{U}_n-\mathfrak{B}\tilde{U}_n,
$$
and $U_n=\dfrac{\tilde{U}_n}{\|\tilde{U}_n\|_{\mathcal{H}}}$, $F_n=\dfrac{f_n}{\|\tilde{U}_n\|_{\mathcal{H}}}$ we conclude that $U_n$ verifies $\|U_n\|_{\mathcal{H}}=1$ and 
\begin{equation*}\label{convresol}
i\gamma_nU_n-\mathfrak{B} U_n=F_n\to 0.
\end{equation*}
Using \eqref{Dissipa} we get that 
$$
U_n\quad \rightarrow\quad0 ,\quad \mbox{strong in}\quad \mathcal{H}.
$$
But this is contradictory to  $\|U_n\|_{\mathcal{H}}=1$. 
Hence our conclusion follows.

\end{proof}

\begin{lemma}\label{inqq1} Under the above notations  there exists a positive constant $c$  such that  
	$$
	\int_{\Omega}\kappa|\nabla\Delta \mathbf{u}|^2 d\Omega\leq c|\gamma|\|U\|_{\mathcal{H}}^2+c\|U\|_{\mathcal{H}}\|F\|_{\mathcal{H}}.
	$$

\end{lemma}
	\begin{proof}Multiplying equation \eqref{xres2} by $\overline{\Delta \mathbf{u}}$, recalling the boundary conditions  and integrating by parts
	 we get 
	
\begin{equation}\label{Ident}
	\underbrace{i\gamma \int_{\Omega}\mathbf{v}\overline{\Delta \mathbf{u}}\;d\Omega+\mu \int_{\Omega}\mathbf{v}\overline{\Delta \mathbf{u}}\;d\Omega}_{\leq c|\gamma|\|U\|_{\mathcal{H}}^2}\underbrace{+\int_{\Omega}(\kappa\Delta ^2\mathbf{u}-\beta\Delta\vartheta)\overline{\Delta \mathbf{u}}\;d\Omega}_{-\int_{\Omega}\kappa|\nabla\Delta \mathbf{u}|^2 d\Omega+\int_{\Omega}\beta\nabla\vartheta\;\overline{\nabla\Delta \mathbf{u}}\;d\Omega}=\underbrace{\int_{\Omega}f_2\overline{\Delta \mathbf{u}}\;d\Omega}_{\leq c\|U\|_{\mathcal{H}}\|F\|_{\mathcal{H}}},
\end{equation}

	for $\gamma$ large. Using the Cauchy–Schwarz inequality we see
\begin{equation}\label{Ident2}
	\int_{\Omega}\left| \beta\nabla\vartheta\;\overline{\nabla\Delta \mathbf{u}}\right| \;d\Omega\leq \frac \kappa 2\int_{\Omega}|\nabla\Delta\mathbf{ u}|^2 d\Omega+\int_{\Omega}|\nabla\vartheta|^2 d\Omega .
\end{equation}
	Using inequality \eqref{Dissipa} and \eqref{Ident2} inside relation \eqref{Ident} we arrive to 
	$$
	\int_{\Omega}\kappa|\nabla\Delta\mathbf{ u}|^2 d\Omega\leq c|\gamma|\|U\|_{\mathcal{H}}^2+c\|U\|_{\mathcal{H}}\|F\|_{\mathcal{H}}.
	$$
	Hence our conclusion follows.
		\end{proof}

\begin{lemma}\label{Lem1}
	Under the above notations  for any $ \epsilon>0$ there exists a positive constant  $c_\epsilon$ such that 
 \begin{eqnarray}\label{ini}
 \int_\Omega |\gamma \Theta_n|^2\;d\Omega&\leq &\epsilon |\gamma|^2\|U\|_{\mathcal{H}}^2+c_\epsilon \|F\|_{\mathcal{H}}^2.\\
\int_{\Omega }|\gamma \Delta \mathbf{u}|^2\;d\Omega&\leq &	\epsilon |\gamma|^2\|U\|_{\mathcal{H}}^2+c_\epsilon\|F\|_{\mathcal{H}}^2.\label{ddu}
\end{eqnarray}

\end{lemma}
	\begin{proof}
Multiplying \eqref{xresn} by $	\overline{i\gamma \Theta_n}$ and taking the real part we have 

\begin{eqnarray*}
\int_\Omega |\gamma \Theta_n|^2\;d\Omega-\mu i\gamma  \|\Theta_n\|^2 \ &=&\dfrac{\beta}{a_n}i\gamma \int_\Omega \nabla \mathbf{ v}\overline{\nabla \Theta_n}\;d\Omega-\sum_{j=0}^{n}\frac{b_j}{a_n}\int_\Omega \nabla\Theta_{j}\overline{i\gamma \nabla\Theta_n}d\Omega\\
&&+\dfrac{\beta}{a_n}\int_\Omega \sum_{j=0}^{n-1}a_j\Theta_{j}\overline{i\gamma \Theta_n}\;d\Omega+\int_\Omega g_n\overline{i\gamma \Theta_n}\;d\Omega.
\end{eqnarray*}
Taking the real part and using \eqref{xres3} we get
\begin{eqnarray*}
\int_\Omega |\gamma \Theta_n|^2\;d\Omega&=&\mbox{Re } \dfrac{\beta}{a_n}i\gamma \int_\Omega \nabla \mathbf{v}\overline{\nabla \Theta_n}\;d\Omega
-\mbox{Re } \sum_{j=1}^{n}\frac{b_j}{a_n}\int_\Omega \nabla\Theta_{j}\overline{\nabla\Theta_n}d\Omega\\
&&-\mbox{Re } \sum_{j=1}^{n}\frac{b_j}{a_n}\int_\Omega \nabla g_{j}\overline{\nabla\Theta_n}d\Omega
+\mbox{Re } \dfrac{\beta}{a_n}\int_\Omega \sum_{j=1}^{n-1}a_j\Theta_{j}\overline{i\gamma \Theta_n}\;d\Omega\\
&&+
\mbox{Re } \dfrac{\beta}{a_n}\int_\Omega \sum_{j=1}^{n-1}a_j\Theta_{j}\overline{i\gamma \Theta_n}\;d\Omega
+\mbox{Re } \int_\Omega g_n\overline{i\gamma \Theta_n}\;d\Omega.
\end{eqnarray*}
%
Using the intermediate derivative theorem, the elliptic regularity, relation \eqref{res1} and \eqref{res3} we have 
\begin{eqnarray}
	\int_\Omega |\nabla \mathbf{v}|^2\;d\Omega&\leq &c\left(\int_\Omega |\mathbf{v}|^2\;d\Omega\right)^{1/2}  \left(\int_\Omega |\Delta \mathbf{v}|^2\;d\Omega\right)^{1/2}\nonumber  \\
	&\leq &c\|\mathbf{v}\| \|(i\gamma +\mu)\Delta \mathbf{u}-\Delta f_1\| \nonumber\\
	&\leq &c|\gamma\|\mathbf{v}\| \|\Delta\mathbf{u}\|+  
	c\|\mathbf{v}\|\|F\|_{\mathcal{H}},\label{estxxx}
\end{eqnarray}
for $\gamma$ large. So we obtain
\begin{eqnarray*}
\left|\beta i\gamma \int_\Omega \nabla \mathbf{v}\overline{\nabla \Theta_n}\;d\Omega\right|&\leq &c|\gamma|\|\nabla \mathbf{v}\|\|U\|_{\mathcal{H}}^{1/2}\|F\|_{\mathcal{H}}^{1/2}\\
&\leq &c|\gamma|\left(|\gamma|^{1/2}\|\mathbf{v}\|^{1/2} \|\Delta \mathbf{u}\|^{1/2}+  c\|U\|_{\mathcal{H}}^{1/2}\|F\|_{\mathcal{H}}^{1/2}\right)\|U\|_{\mathcal{H}}^{1/2}\|F\|_{\mathcal{H}}^{1/2}\\
&\leq &c|\gamma|\left(|\gamma|^{1/2}\|U\|_{\mathcal{H}}+  c\|U\|_{\mathcal{H}}^{1/2}\|F\|_{\mathcal{H}}^{1/2}\right)\|U\|_{\mathcal{H}}^{1/2}\|F\|_{\mathcal{H}}^{1/2}\\
&\leq &c|\gamma|^{3/2}\|U\|_{\mathcal{H}}^{3/2}\|F\|_{\mathcal{H}}^{1/2}+c\|U\|_{\mathcal{H}}\|F\|_{\mathcal{H}}.
\end{eqnarray*}
 So, we see that 
 
 \begin{equation*}\label{ini2}
 \int_\Omega |\gamma \Theta_n|^2\;d\Omega\leq \epsilon |\gamma|^2\|U\|_{\mathcal{H}}^2+c_\epsilon \|F\|_{\mathcal{H}}^2.
 \end{equation*}
 Hence inequality \eqref{ini} follows. Finally,  
from \eqref{xresn} we have 
 \begin{equation}\label{inires}
	i\gamma \Theta_n-\dfrac{1}{a_n}\left( \beta i\gamma \Delta \mathbf{u}+\sum_{j=0}^{n}b_j\Delta\Theta_{j}\right) -\dfrac{1}{a_n}\sum_{j=0}^{n-1}a_j\Theta_{j}=R_1,
 \end{equation} 
where 
$$
R_1=g_n-\mu\Theta_n +\dfrac{1}{ca_n}\beta \mu\Delta\mathbf{ u} + \dfrac{1}{ca_n}\beta \Delta f_1.
$$
Multiplying \eqref{inires} by $\overline{i\gamma \Delta \mathbf{u}}$ we get
\begin{eqnarray}\nonumber
\dfrac{\beta }{a_n}\int_{\Omega }|\gamma \Delta \mathbf{u}|^2\;d\Omega&=&\underbrace{i\gamma \int_{\Omega }\Theta_n\overline{i\gamma \Delta \mathbf{u}}\;d\Omega}_{\leq c_\epsilon\|\gamma\Theta_n\|^2+\epsilon\|\gamma\Delta \mathbf{u}\|^2}
-\underbrace{\int_{\Omega }\sum_{j=0}^{n}\dfrac{b_j}{a_n}\Delta\Theta_{j}\overline{i\gamma \Delta \mathbf{ u}}\;d\Omega}_{:=J}\\
	&&+\underbrace{\dfrac{1}{a_n}\beta\int_{\Omega }\sum_{j=0}^{n-1}a_j\Theta_{j}\overline{i\gamma \Delta\mathbf{ u}}d\Omega}_{\leq c\|U\|_{\mathcal{H}}\|F\||_{\mathcal{H}}+c_\epsilon\|\Theta_n\|^2+\epsilon\|\gamma\Delta \mathbf{u}\|^2}-\underbrace{\int_{\Omega }R_1\overline{i\gamma \Delta\mathbf{ u}}\;d\Omega}_{\leq c_\epsilon\|F\|_{\mathcal{H}}^2+\epsilon\|\gamma\Delta \mathbf{u}\|^2}.\label{jjj}
\end{eqnarray}
Poincare's inequality implies that  $\|\Theta_j\|^2\leq c\|U\|_{\mathcal{H}}\|F\|_{\mathcal{H}} $ for $j=0,\cdots,n-1$. To show inequality \eqref{ddu} we only need to estimate $\left| J\right| $. Indeed, 
$$
|J|=\left|
\int_{\Omega }\sum_{j=0}^{n}\frac{b_j}{a_n}\nabla\Theta_{j}\overline{i\gamma \nabla\Delta \mathbf{u}}\;d\Omega
\right|\leq c\sum_{j=0}^n|\gamma|\int_{\Omega }|\nabla\Theta_{j}|^2\;d\Omega+\epsilon|\gamma|\int_{\Omega }|\nabla\Delta\mathbf{ u}|^2\;d\Omega.
$$
From Lemma \ref{inqq1} we get 
\begin{eqnarray*}
\left|
J
\right|
\leq 
\epsilon |\gamma|^2\|U\|_{\mathcal{H}}^2+c_\epsilon\|F\|_{\mathcal{H}}^2.
\end{eqnarray*}
So we have
\begin{eqnarray*}
\int_{\Omega }|\gamma \Delta \mathbf{u}|^2\;d\Omega&\leq &	\epsilon |\gamma|^2\|U\|_{\mathcal{H}}^2+c_\epsilon\|F\|_{\mathcal{H}}^2.
\end{eqnarray*}
Inserting the above inequalities into \eqref{jjj} our conclusion follows. 

\end{proof}

\begin{theorem}
	The operator  $\mathfrak{B}$ generates  an analytic semigroup.
\end{theorem}

	\begin {proof}
Multiplying  equation \eqref{xres2} by $\overline{i\gamma \mathbf{v}}$ and using \eqref{xres1} we get

$$
i\gamma \mathbf{v}+\mu \mathbf{v}+  \kappa\Delta ^2\mathbf{u}-\beta\Delta \vartheta =f_2,
$$
 \begin{align}
 \int_{\Omega}|\gamma \mathbf{v}|^2-i\gamma \mu |\mathbf{v}|^2\;d\Omega=\underbrace{\dfrac{1}{\rho}\int_{\Omega}\kappa i\gamma \Delta\mathbf{ u} \overline{\Delta \mathbf{v}}\;d\Omega}_{\leq c\gamma^2\|\Delta \mathbf{u}\|^2+c\|F\|_\mathcal{H}^2}+\underbrace{\beta\int_{\Omega}i\gamma \nabla \vartheta\overline{\nabla \mathbf{v}}\;d\Omega}_{:=J_1}+\underbrace{\int_{\Omega}f_2\overline{i\gamma \mathbf{v}}d\Omega}_{\leq \epsilon \gamma^2\|\mathbf{v}\|^2+c\|F\|_\mathcal{H}^2}.\label{qres2}
\end{align}
Using \eqref{estxxx}  and \eqref{Dissipa} we get that 
\begin{equation*}\label{jooo}
J_1\leq \epsilon|\gamma|^2\|U\|^2_{\mathcal{H}}+c_\epsilon\|F\|^2_{\mathcal{H}}.
\end{equation*}
From Lemma \ref{Lem1} and taking the real part in \eqref{qres2} we get 

 \begin{align}
 \int_{\Omega}|\gamma \mathbf{v}|^2\;d\Omega\leq \epsilon|\gamma|^2\|U\|^2_{\mathcal{H}}+c_\epsilon\|F\|^2_{\mathcal{H}}\label{xxg}.
\end{align}
Moreover, by definition we have that 
\begin{equation}\label{xyxy}
\sum_{j=0}^{n-1} \int_{\Omega}\left|\gamma\nabla \Theta_{j} \right|^2d\Omega\leq c\sum_{j=1}^{n} \int_{\Omega}\left|\gamma\nabla \Theta_{j} \right|^2d\Omega+c\|F\|_{\mathcal{H}}^2.
\end{equation}
From \eqref{ddu}, \eqref{ini},  \eqref{xxg} and \eqref{xyxy} and recalling the definition of the norm of  $U$ we conclude that 
$$
|\gamma|^2\|U\|^2_{\mathcal{H}}\leq \epsilon |\gamma|^2\|U\|^2_{\mathcal{H}}+c_\epsilon\|F\|^2_{\mathcal{H}}.
$$
From where our conclusion follows.

	\end{proof}

\noindent 
As a consequence we have 
\begin{theorem}
The semigroup generated by the operator $\mathcal{A}$ is analytic.
\end{theorem}

\section{Local heat conduction}\label{Gevrey}
Here we consider $\Omega=\Omega_1\cup \Omega_2\subset \mathbb{R}^2$ an open set such that  the thermal effect is efective only over $\Omega_1$ (see figure 1). Our main result is to prove that the operator $\mathfrak{B}$ given in \eqref{opeA1} generates  a Gevrey  semigroup of class 4 for $t>0$

With the same notations as in the sections above the corresponding model is given by 
\begin{align}
	\rho \widehat{u}_{tt}&=-\kappa _1\Delta ^2 \widehat{u}-\beta \Delta \widehat{\theta }& \text{in} \ \Omega_1 \times \mathbb{R}^{+}_{0}, \label{Lb1}\\
	c_T\dfrac{\partial}{\partial t}\left(\sum_{j=0}^na_j\theta^{(j)}\right)&= \left(\sum_{j=0}^nb_j\Delta \theta^{(j)}\right) +\beta \Delta \widehat{u}_t& \text{in} \ \Omega_1 \times \mathbb{R}^{+}_{0},\label{Lb2}\\
	\rho \widehat{v}_{tt}&=-\kappa_2 \Delta ^2 \widehat{v} & \text{in} \ \Omega_2 \times \mathbb{R}^{+}_{0}, \label{Lb3}
\end{align}	
Here $\widehat{u}$ is the displacement, $\vartheta= a_0\theta+a_1\theta^{(1)}+...+a_n\theta^{(n)}$ which is the temperature effective only in $\Omega_1$.
We adjoin the boundary conditions 
\begin{align}\label{fronteraloc}
	\widehat{u}=\Delta \widehat{u}=\widehat {\theta}=0  \ \ \text{on} \  \ \Gamma_0,\quad \widehat {\theta}=0 \ \ \text{on} \  \  \Gamma_1, \ \ t>0 .
\end{align}
Verifying the transmission conditions 
\begin{equation}\label{trans1}
	\widehat{u}(x,t)=\widehat{v}(x,t),\quad \frac{\partial \widehat{u}}{\partial\nu}(x,t)=\frac{\partial \widehat{v}}{\partial\nu}(x,t) \ \ \text{on} \  \ \Gamma_1,
\end{equation}
\begin{equation}\label{trans2}
	\kappa_1 \Delta \widehat{u}+\beta\widehat{\theta} =\kappa_2 \Delta\widehat{v} ,\quad \kappa_1\frac{\partial \Delta \widehat{u}}{\partial\nu}+\beta \frac{\partial \widehat{\theta}}{\partial\nu}=\kappa_2\frac{\partial \Delta \widehat{v}}{\partial\nu} \ \ \text{on} \  \ \Gamma_1.
\end{equation}
Additionally, we consider the following initial conditions

\begin{align}
	\widehat{u}(x,0)=\widehat{u}_0(x),\;\; \widehat{u}_t(x,0)=\widehat{u}_1(x), \quad \theta^{(j)}(x,0)=\eta_j(x) \label{Lb33n2}
\end{align}
\begin{equation}
\widehat{v}(x,0)=\widehat{v}_0(x),\;\; \widehat{v}_t(x,0)=\widehat{v
}_1(x).\label{Lb33xx}
\end{equation}
The total energy associated with the system \eqref{Lb1}-\eqref{Lb33xx} is defined by

\begin{align*}
2	E(t)&=\int_{\Omega_1 }  \left| \widehat{u}_{t}\right|^2+ \kappa_1\left|\Delta  \widehat{u} \right| ^2+ \sum_{j=0}^{n-1}\left|\nabla\theta^{(j)}\right| ^2 + \left|\widehat{\theta}\right|^2  d\Omega+\int_{\Omega_2 }   \left|\widehat{v}_t\right| ^2+ \kappa_2\left| \Delta \widehat{v} \right| ^2  d\Omega.
\end{align*}
As in section 2, we have 
\begin{align*}
	\dfrac{d}{dt}E(t)\leq -\frac{a_nb_n}{2}\int_{\Omega _1}|\nabla\theta^{(n)}|^2\; d\Omega+c\int_{\Omega _1}  \sum_{j=0}^{n-1}| \nabla\theta^{(j)}|^2\;d\Omega.
\end{align*}
The system once more,  is not dissipative in general.

\subsection{Existence: Semigroup approach}

	As in section 2 we considerar $\rho=c_T=1$. 	Let us introduce the following notations 
				\begin{eqnarray*}
			\mathds{H}^m(\Omega)&=&\left\lbrace (\mathbf{u},\mathbf{v})^T\in H^m (\Omega_1)\times H^m(\Omega_2)\right\rbrace .\\
			\mathds{H}^0(\Omega)&=&\left\lbrace (\mathbf{u},\mathbf{v})^T\in L^2 (\Omega_1)\times L^2(\Omega_2)\right\rbrace .\\
						\mathds{H}^2_\Gamma(\Omega)&=&\left\lbrace (\mathbf{u},\mathbf{v})^T\in H^2 (\Omega_1)\times H^2(\Omega_2),  \text{verifying \eqref{trans1}} \ \ \text{on} \ \Gamma_1 \right\rbrace .
		\end{eqnarray*}
Under the above notation we define the phase space 
		\begin{align*}
			\mathcal{H}=\mathds{H}^2_{\Gamma}(\Omega)\times \mathds{H}^0(\Omega)\times [H^1_0(\Omega_1)]^n\times [L^2(\Omega_1)]^n,
		\end{align*}
		which is a Hilbert space with the norm
	\begin{align*}
	\left\|U \right\|_{\mathcal{H}}^2=\int_{\Omega_1 } \left|\mathbf{w}\right|^2+ \kappa_1\left|\Delta \mathbf{u} \right| ^2+  \sum_{j=0}^{n-1}\left|\nabla \Theta_j\right| ^2 + \left|\sum_{j=0}^{n}a_i\Theta_i\right|^2   d\Omega+\int_{\Omega_2 } \left| \mathbf{z} \right| ^2+ \kappa_2\left| \Delta \mathbf{v} \right| ^2 d\Omega, 
\end{align*}
where $U=(\mathbf{u},\mathbf{v},\mathbf{w},\mathbf{z},,\Theta_0,\Theta_1,...,\Theta_n)$. Denoting $ \widehat{u}_t=\mathbf{w}, \ \widehat{v}_t=\mathbf{z}$  problem \eqref{Lb1}-\eqref{Lb33xx} can be written as 
$$\dfrac{dU}{dt}=AU, \   U(0)=U_0.$$
where 
\begin{align}
	\mathcal{A}U=\begin{pmatrix}
	\mathbf{w}\\
	\mathbf{	z}\\
		-\Delta(\kappa_1\Delta \mathbf{u} +\beta(a_0\Theta_0+a_1\Theta_1+...+a_n\Theta_n))\\
		-\kappa_2 \Delta^2 \mathbf{v}\\
		\Theta_1 \\
		\Theta_2\\
		.\\
		.\\
		.\\
		\Theta_n\\
	\displaystyle	\dfrac{\beta}{ a_n}\Delta \mathbf{v}+\dfrac{1}{a_n} \sum_{j=0}^{n}b_j\Delta\Theta_j-\dfrac{1}{a_n}\sum_{j=0}^{n-1}a_j\Theta_{j}
	\end{pmatrix} . 
	\label{opeA1}
\end{align}
with
\begin{align*}
	D(\mathcal{A})=\left\lbrace U\in \mathcal{H}: (\mathbf{u},\mathbf{v})^T\in  \mathds{H}_{\Gamma}^2 ,(\kappa_1\Delta \mathbf{u}+\beta \vartheta,\kappa_2\Delta \mathbf{z})^T\in  \mathds{H}_{\Gamma}^2,\;\;\sum_{j=0}^{n}b_j \Theta_j\in H^2(\Omega _1) \right\rbrace .
\end{align*}


Using the same reasoning as in the first problem we have that under the above condition the operator $ \mathcal{A}$ is not dissipative, in fact 
\begin{align*}
	\mbox{Re }\left\langle \mathcal{A}U,U \right\rangle =	\mbox{Re } \int_{\Omega_1}\left( \sum_{k=0}^{n-1}\nabla \Theta_{k+1}\overline{\nabla\Theta}_k-\sum_{k,\ell=0}^{n}a_k b_\ell\nabla \Theta_{\ell}\overline{\nabla\Theta}_{k}\right) \ d\Omega.
\end{align*}
Using similar reasoning as in \eqref{EnergIn}.
\begin{align*}
	\mbox{Re }\left\langle \mathcal{A}U,U \right\rangle\leq -\frac{a_nb_n}{2}\int_{\Omega_1 }|\nabla\Theta_{n}|^2\; d\Omega+c_0\int_{\Omega _1}  \sum_{j=0}^{n-1}| \nabla\Theta_{j}|^2\;d\Omega.
\end{align*}
To apply the semigroup theory of dissipative generators, we consider a continuous perturbation of $\mathcal{A}$ given by $\mathfrak{B}=\mathcal{A}-2c_0\mathbf{I}$. It is clear that 
$D(\mathfrak{B})=D(\mathcal{A})$ and 
\begin{align}\label{nondis1}
	\mbox{Re }\left\langle \mathfrak{B}U,U \right\rangle\leq -\frac{a_nb_n}{2}\int_{\Omega_1 } \left| \nabla\Theta_{n}\right| ^2\; d\Omega-c_0\left\|U \right\|^2 .
\end{align}

Hence to show that $\mathfrak{B}$ is the infinitesimal generator of a contraction semigroup, it is enough to show that $0\in \varrho(\mathfrak{B})$. Following the same line of reasoning as that used in the proof of Theorem \ref{Existe1} we can establish.

\begin{theorem}\label{Existe2}
	The operator $\mathcal{A}$ generates a quasi-contractive semigroup.
\end{theorem}

\noindent
We conclude this section by establishing the Theorem of Intermediate Derivatives, whose proof can be seen in  \cite{Adams}.

\begin{proposition}\label{p8}
	
	Let $\Omega\subset \mathbb{R}^N$ an open set then exist a constant $\epsilon_0>0$ and a constant $K(\epsilon_0,m,p,\Omega)$, such that for all $ 0<\epsilon<\epsilon_0$ and $0<j<m$ that satisfying
	\begin{align*}
	\left\|D^j u \right\|_{L^p(\Omega)} \leq K \left( \epsilon^{m-j}\left\|D^m u \right\|_{L^p(\Omega)} +\epsilon ^j\left\|u \right\|_{L^p(\Omega)} \right) .
	\end{align*}
	By $\left\| D^j u\right\| _{L^p(\Omega)}$ we are denoting the norm in $L^p$ of all derivatives of order $j$.
\end{proposition}
\begin{proof}
	See \cite{Adams} (see also \cite{Meqd} for the one dimensional case). 
\end{proof}

\subsection{Gevrey's class for the radial symmetrical  plate model }
 
Here we consider symmetrical and radial solutions. 
Let us denote by $O(2)$ the set of orthogonal $n\times n$ real
matrices and by $SO(2)$ the set of matrices in $O(2)$ which have
determinant 1.
\begin{lemma} \label{le2.2}
 Assume that the initial data  of problem \eqref{Lb1}-\eqref{Lb33xx}satisfies 
\begin{eqnarray}
 \begin{array}{ll}
\widehat{u}_0(Gx)=  \widehat{u}_0({{x}}),\; \widehat{u}_1(G \,{{x}})=  \widehat{u}_1({{x}}),\; \eta_j(Gx,0) 
  = \eta_j({{x}}), & \forall\; {{x}}\in\bar \Omega_1,  \\
  \noalign{\medskip}
\widehat{v}_0(G\, {{x}})=  \widehat{v}_0({{x}}),\; \widehat{v}_1(G \,{{x}})=  \widehat{u}_1({{x}}),
  & \forall\; {{x}}\in\bar \Omega_0,  \\
  \noalign{\medskip}
  G\in O(2)\quad\text{if}\quad n=2\quad\text{and}\quad
  G\in SO(n)\quad\text{if}\quad n\geq 3. & \label{2.5}
 \end{array}
\end{eqnarray}
Then the corresponding solution $(\widehat{u},\widehat{v},\widehat{\theta})$  verifies 
\begin{eqnarray}
 \begin{array}{llll}
\widehat{v}({{x}},t)=\phi (r,t),&  
  &\forall\; {{x}}\in \Omega_1,&t\geq 0, \\
  \noalign{\medskip}
  \theta({{x}},t)=\psi (r,t), & & \forall\; {{x}}\in\Omega_1,&t\geq 0, \\
  \noalign{\medskip}
\widehat{v}({{x}},t)=\,\eta (r,t),& 
  &\forall\; {{x}}\in\Omega_0,&t\geq 0,
 \end{array} \label{2.8}
\end{eqnarray}
where $r=|{{x}}|$, for some functions $\phi$, $\psi$, $\eta$.
\end{lemma}
\begin{proof}
The proof is inmediate. 
\end{proof}

Denoting by $\mathcal{H}_R$ the space $\mathcal{H}$ for radial symmetrical function, thanks to Lemma \ref{le2.2} we conclude that 
the semigroup $S(t)=e^{\mathcal{A}t}$ is invariant over $\mathcal{H}_R$, that is to say $S(t)\mathcal{H}_R\subset \mathcal{H}_R$ and the corresponding infinitesimal generator is given by $\mathcal{A}$ defined over the domain $D_R(\mathcal{A})=\mathcal{H}_R\cap D(\mathcal{A})$

\begin{theorem} \label{theoGe}
	Let $S(t)=e^{\mathcal{A}t}$ be a contraction semigroup on a Hilbert space $X$.  Suppose that the infinitesimal generator $\mathcal{A}$ satisfies 
	\begin{eqnarray*} \label{gege}
		i\mathbb{R}\subset \rho(\mathcal{A}),\quad \mbox{and }\quad   \lim_{\lambda \in \mathbb{R}, \ |\lambda| \rightarrow \infty} \sup |\lambda|^{\varsigma}\|(i \lambda - \mathcal{A})^{-1} \|_{\mathcal{L}(\mathcal{H})} < \infty, 
	\end{eqnarray*}
	for some $0< \varsigma<1$. 
	Then, $S(t)$ is of Gevrey's class $\dfrac{1}{\varsigma}$ for $t > 0$.
\end{theorem}

	We write the resolvent equation, $i\lambda  U-\mathcal{\mathfrak{B}}U=F$ as its components as follows
	\begin{align}
i\lambda \mathbf{u}+\mu \mathbf{u}-	\mathbf{w}=	f_1,\label{f11}\\
i\lambda  \mathbf{v}+\mu \mathbf{v}-	\mathbf{z}=	f_2,\label{f22}\\
i\lambda  \mathbf{w}+\mu \mathbf{w}+\Delta(\kappa_1\Delta \mathbf{u} +\beta(a_0\Theta_0+a_1\Theta_1+...+a_n\Theta_n))=	f_3, \label{f33}\\
i\lambda  \mathbf{z}+\mu \mathbf{z}+\kappa_2 \Delta^2 \mathbf{v}=f_4,\label{f44}\\
	i\lambda \Theta_k+\mu \Theta_k-\Theta_{k+1}=g_k,\label{gnn}\\
i\lambda   \Theta_n+\mu \Theta_n-	\dfrac{\beta}{ a_n}\Delta\mathbf{w}+\dfrac{1}{a_n} \sum_{j=0}^{n}b_j\Delta\Theta_j-\dfrac{1}{a_n}\sum_{j=0}^{n-1}a_j\Theta_{j}=g_{n+1}. \label{gn}
		\end{align}
where  $k=0\;,\cdots \; n-1$. 

\begin{remark}
Let us denote by 
$$
\Phi=\kappa_1\Delta \mathbf{u} +\beta\vartheta,\quad \Psi=\kappa_2\Delta \mathbf{v} ,
$$
then from \eqref{f33} and \eqref{gn}  we have that
\begin{eqnarray*}
\Delta\Phi&=&	f_3-i\lambda  \mathbf{w}-\mu \mathbf{w}\quad \in L^2(\Omega_1),\\
\Delta\Psi&=&	f_4-i\lambda  \mathbf{z}-\mu \mathbf{z}\quad  \in L^2(\Omega_2).
\end{eqnarray*}
For $U=(\mathbf{u},\mathbf{v},\mathbf{w},\mathbf{z},,\Theta_0,\Theta_1,...,\Theta_n)\in D(\mathcal{A})$ we have that $\Phi$ 
satisfies  the boundary condition 
$$
\Phi=0,\quad \mbox{on }\quad \Gamma_0,
$$
and also the transmission conditions 
$$
\Phi=\Psi,\quad \frac{\partial\Phi}{\partial\nu}=\frac{\partial\Psi}{\partial\nu},\quad \mbox{on }\quad \Gamma_1.
$$
By the elliptic regularity to second order transmission problems we have that $\Phi,\Psi\in H^2$ and we have that 
$$
\|\Phi\|_{H^2}\leq c\|\Delta\Phi\|_{L^2},\quad \|\Psi\|_{H^2}\leq c\|\Delta\Psi\|_{L^2}.
$$
\hfill $\Box$
\end{remark}

		\begin{lemma}\label{TrazoInter}
For any function $V\in H^{1}(\Omega) $ we have 
\begin{equation}\label{TrazH1}
\|V\|_{L^2(\partial \Omega)} \leq c \|V\|_{L^{2}(\Omega)}^{1/2} \|\nabla V\|_{L^{2}(\Omega)} ^{1/2}.
\end{equation}
Moreover if $V\in H^{2}(\Omega) $ is a radial function  we have 
\begin{equation}\label{TrazH2}
\|V\|_{L^2(\partial \Omega)} \leq c \|V\|_{L^{2}(\Omega)}^{3/4} \|\Delta V\|_{L^{2}(\Omega)} ^{1/4},
\end{equation}
\begin{equation}\label{TrazH3}
\|\nabla V\|_{L^2(\Omega)} \leq c \|V\|_{L^{2}(\Omega)}^{1/2} \|\Delta V\|_{L^{2}(\Omega)} ^{1/2}.
\end{equation}

				\end{lemma}
				\begin{proof} Using the trace Theorem and Theorem \ref{p8}
				$$
\|V\|_{L^2(\partial \Omega)} \leq c\|V\|_{H^{1/2}(\Omega)} \leq c \|V\|_{L^{2}(\Omega)}^{1/2} \|\nabla V\|_{L^{2}(\Omega)} ^{1/2}.
$$
Hence inequality \eqref{TrazH1} follows. 
If $V$ is a radial function we have that 
$$
\|V\|_{L^2(\partial \Omega)} =2\pi R_0|V(R_0)|\leq c\|V\|_{L^2(R_0,R)}\|V'\|_{L^2(R_0,R)}\leq c\|V\|_{L^2(R_0,R)}^{3/4}\|V''\|_{L^2(R_0,R)}^{1/4}.
$$
Since for any radial function we have that $\Delta V=V''$, hence we arrive to 
$$
\|V''\|_{L^2(R_0,R)}\leq c\|\Delta V\|_{L^2(\Omega)}.
$$
From the two above inequalities we get \eqref{TrazH2}. Finally, since 
$V$ is a radial function we have that 
$$
\frac{\partial V}{\partial x}=V'(r)\frac{x}{r},\quad \frac{\partial V}{\partial y}=V'(r)\frac{y}{r}.
$$
Hence we have that 
$$
\|\nabla V\|_{L^2(\Omega)} = \| V'\|_{L^2(\Omega)}. 
$$
Using a change of variable we get 
$$
 \| V'\|_{L^2(\Omega)}\leq 2\pi  \| V'\|_{L^2(R_0,R)}\leq c \| V\|_{L^2(R_0,R)}^{1/2} \| V''\|_{L^2(R_0,R)}^{1/2}.
 $$
Using the identity  $\Delta V=V''$ inequality \eqref{TrazH3} follows. 
				\end{proof}

\bigskip

\noindent  
To facilitate our notations let us introduce the  following functional,
	\begin{align*}
\mathfrak{E}^2=\int_{\Omega_1}|\mathbf{w}|^2+\kappa_1 |\Delta \mathbf{u}|^2+ \left|\Theta_n\right|^2 d\Omega,\quad \mathfrak{R}^2=\left\|U \right\|_\mathcal{H}\left\|F \right\|_\mathcal{H}+\left\|F \right\|^{2}_\mathcal{H}, \quad 
\mathfrak{R}_0^2= \left\|U \right\|_\mathcal{H}\left\|F \right\|_\mathcal{H}.
	\end{align*}

\begin{lemma} \label{uthetapp}
Over the thermoelastic part, for any $ \epsilon>0$ there exists a positive constant $c$ and $c_\epsilon$ such that 
\begin{eqnarray}
\int_{\Omega_1} |\Theta_n|^2d\Omega&\leq &\frac{c}{|\lambda|^{1/2}}\|\mathbf{w}\|^{1/2}\|\Delta\mathbf{u}\|^{1/2}\mathfrak{R}_0 +\frac{c}{|\lambda|}\mathfrak{R}^2,\label{thetp}\\
\int_{\Omega_1} |\Theta_n|^2d\Omega&\leq &\epsilon \|\mathbf{w}\| \|\Delta\mathbf{u}\| +\frac{c_\epsilon}{|\lambda|}\mathfrak{R}^2.\label{thetp2} 
\end{eqnarray}
\end{lemma}	 
\begin{proof}
Multiplying equation \eqref{gn} by $\overline{i\lambda \Theta_n}$  we get

\begin{align}
\int_{\Omega_1} |\lambda  \Theta_n|^2d\Omega&=\underbrace{\dfrac{\beta}{ a_n}\int_{\Omega_1}\Delta \mathbf{w} \overline{i\lambda \Theta_n} \ d\Omega}_{:=J_2}+\underbrace{ \dfrac{1}{a_n}\int_{\Omega_1}\sum_{j=0}^{n}b_j\nabla\Theta_j\overline{i\lambda \nabla\Theta_n} \ d\Omega}_{=\int_{\Omega_1}\sum_{j=0}^{n-1}b_j\nabla\Theta_j\overline{\nabla\Theta_n}\ d\Omega+b_ni\lambda\|\nabla\Theta_n\|^2 }\nonumber \\
&+\dfrac{1}{a_n}\underbrace{\int_{\Omega_1}\sum_{j=1}^{n-1}a_j\Theta_{j} \overline{i\lambda \Theta_n} \ d\Omega}_{\leq c\|U\|_{\mathcal{H}}\|F\|_{\mathcal{H}}+\epsilon\|\lambda \Theta_n\|^2 }+\underbrace{\int_{\Omega_1}g_{n+1} \overline{i\lambda \Theta_n}\ d\Omega}_{\leq c\|F\|_{\mathcal{H}}^2+\epsilon\|\lambda \Theta_n\|^2}-\int_{\Omega_1}i \lambda \mu \left| \Theta_n \right|^2 \ d \Omega. \label{gxn}
\end{align}
Using the same procedure as in \eqref{estxxx} we get
\begin{eqnarray} 
\int_{\Omega_1} |\nabla \mathbf{w}|^2\;d\Omega
&\leq &c|\lambda|\|\mathbf{w}\|\|\Delta\mathbf{u}\|+  
c\|U\|_{\mathcal{H}}\|F\|_{\mathcal{H}}.\label{estvvv1}
\end{eqnarray}
Using \eqref{estvvv1} we get
\begin{eqnarray*}
|J_2| &=&\left|\dfrac{\beta}{ a_n}\int_{\Omega_1}\nabla \mathbf{w} \overline{i\lambda \nabla\Theta_n} \ d\Omega\right|\\
&\leq &c|\lambda|\left(|\lambda|^{1/2}\|\mathbf{w}\|^{1/2}\|\Delta\mathbf{u}\|^{1/2}+  c\|U\|_{\mathcal{H}}^{1/2}\|F\|_{\mathcal{H}}^{1/2}\right)\|U\|_{\mathcal{H}}^{1/2}\|F\|_{\mathcal{H}}^{1/2}\\
&\leq &c|\lambda|^{3/2}\|\mathbf{w}\|^{1/2}\|\Delta\mathbf{u}\|^{1/2}\mathfrak{R}_0 +c_\epsilon|\lambda|\mathfrak{R}_0^2.
\end{eqnarray*}
Taking the real part in identity \eqref{gxn} and using the above inequalities we get 
$$
\int_{\Omega_1} |\lambda \Theta_n|^2 \ d\Omega\leq c|\lambda|^{3/2}\|\mathbf{w}\|^{1/2}\|\Delta\mathbf{u}\|^{1/2}\mathfrak{R}_0 +c_\epsilon|\lambda|\|U\|_{\mathcal{H}}\|F\|_{\mathcal{H}} +c_\epsilon\|F\|_{\mathcal{H}}^2.
$$
 Hence relation \eqref{thetp} follows. 
\end{proof}

\bigskip 

\noindent 
We begin our procedure rewriting equation \eqref{gn} and
using  \eqref{f11}

\begin{equation}\label{eta}
\Theta_n+\dfrac{\mu \Theta_n}{i\lambda}-	\dfrac{\beta}{ a_n}\Delta \mathbf{u}-\dfrac{1}{i\lambda  a_n} \sum_{j=0}^{n}b_j\Delta\Theta_j+\dfrac{1}{a_ni\lambda  }\sum_{j=0}^{n-1}a_j\Theta_{j}=\frac{1}{i\lambda  }g_{n+1} +\dfrac{\beta}{ i\lambda  a_n}\Delta f_1.
\end{equation}
Let us denote by 
				\begin{equation}
				\label{lll}\mathcal{L}=\dfrac{1}{i\lambda a_n} \sum_{j=0}^{n}b_j\nabla\Theta_j.
\end{equation}
Note that, for a sufficiently large positive constant $c$ we see,
				\begin{equation}\label{Ll2}
\|\mathcal{L}\|_{L^{2}(\Omega_1)}\leq \frac{c}{|\lambda|}\mathfrak{R}.
				\end{equation}

		\begin{lemma}\label{etanTT}
Under the above notations we have that
$$
\int_{\Gamma _1}\left| \mathcal{L}\right| ^2 d\Gamma \leq \frac{c}{|\lambda |}\mathfrak{E}^2+\frac{c}{|\lambda |^2}\|U\|_{\mathcal{H}}\|F\|_{\mathcal{H}}, 
$$
$$
\left\|\frac{1}{|\lambda|} \nabla\Theta_{k} \right\|_{\Gamma_1} 
\leq \frac{c}{|\lambda|}\mathfrak{R}^{1/2}\mathfrak{E}^{1/2}+\frac{c}{|\lambda|^2}\mathfrak{R}\quad \mbox{for } k=1,2\cdots n.
$$

				\end{lemma}
				\begin{proof} 
From equation \eqref{eta}  we get
\begin{eqnarray*}
\|div\ \mathcal{L}\|_{L^2(\Omega)}\leq c\|\Theta_n\|+c\|\Delta\mathbf{u}\|+\frac{c}{|\lambda |}\mathfrak{R}.
\end{eqnarray*}
Using  inequality \eqref{thetp2}  of Lemma \ref{uthetapp} into the above relation  we get 
\begin{eqnarray}
\|div\ \mathcal{L}\|_{L^2(\Omega)}
&\leq & c\|\mathbf{w}\|+c\|\Delta\mathbf{u}\|+\frac{c}{|\lambda|^{1/2}}\mathfrak{R} \label {i21}.
\end{eqnarray}
By Lemma \ref{TrazoInter} and relations \eqref {i21}, \eqref{Ll2}
 we get 
\begin{eqnarray}\label {i3}
\|\mathcal{L}\|_{\Gamma_1}^2
&\leq&c\|\mathcal{L}\|_{L^2(\Omega)}\|\mbox{div }\mathcal{L}\|_{L^2(\Omega)}\nonumber\\
&\leq& \frac{c}{|\lambda |}\mathfrak{R}_0\mathfrak{E}+\frac{c}{|\lambda|^2}\|F\|_{\mathcal{H}}^2.
\end{eqnarray}
On the other hand, recalling the definition of $\mathcal{L} $ of \eqref{lll} and using \eqref{gnn} we get
$$
	\Theta_k=\frac{1}{i\lambda +\mu}\left( \Theta_{k+1}+g_k\right) ,
	$$
	where
\begin{eqnarray*}
\frac{b_n}{i \lambda a_n} \nabla\Theta_{n} 
&=&\mathcal{L}-\dfrac{1}{i\lambda a_n} \sum_{j=0}^{n-2}b_j\nabla\Theta_j-\dfrac{b_{n-1}}{i\lambda a_n} \nabla\Theta_{n-1}\\
&=&\mathcal{L}-\dfrac{1}{i\lambda a_n} \sum_{j=0}^{n-2}b_j\nabla\Theta_j-\dfrac{b_{n-1}}{i\lambda a_n(i\lambda +\mu)} \left( \nabla\Theta_{n}+\nabla g_{n-1}\right) .
\end{eqnarray*}
From where we get 
\begin{eqnarray*}
\left( \frac{b_n}{i \lambda a_n} +\dfrac{b_{n-1}}{i\lambda a_n(i\lambda +\mu)} \right) \nabla\Theta_{n}
&=&\mathcal{L}-\dfrac{1}{i\lambda a_n} \sum_{j=0}^{n-2}b_j\nabla\Theta_j-\dfrac{b_{n-1}}{i\lambda a_n\left( i\lambda +\mu\right) } \nabla g_{n-1}.
\end{eqnarray*}
Repeating the above procedure and taking $\lambda$ large we and \eqref{i3} conclude that there exists a positive constant $c$ such that 
\begin{eqnarray*}
\left\|\frac{1}{|\lambda|} \nabla\Theta_{n} \right\|_{\Gamma_1} 
\leq c\left\| \mathcal{L}\right\|_{\Gamma_1} +\frac{c}{|\lambda|^2}\|F\|_{\mathcal{H}}\leq  \frac{c}{|\lambda|^{1/2}}\mathfrak{R}_0^{1/2}\mathfrak{E}^{1/2}+\frac{c}{|\lambda|^2}\mathfrak{R}.
\end{eqnarray*}
From where our conclusion follows. 
 \end{proof}

\begin{lemma}\label{Nnll} Let us denote by $D$ any first order operator, then  we have 
 \begin{eqnarray}
\|\Theta_n\|_{H^2(\Omega_1)} 
&\leq &c\|\mbox{div }\mathcal{L}\|,\nonumber\\
\int_{\Omega_1}\left|\nabla(\kappa_1\Delta  \mathbf{u}+\beta \vartheta)\right|^2 d \Omega
&\leq &c|\lambda|\|\mathbf{w}\|\mathfrak{E}+c\|f_3\|\mathfrak{E},\nonumber
\end{eqnarray}
and over the boundary we have 
 \begin{align}
\left\|\nabla(\kappa_1\Delta  \mathbf{u} +\beta\vartheta)\right\|_{L^2(\partial\Omega_1)}
&\leq c|\lambda|^{3/4}\|\mathbf{w}\|^{3/4}\mathfrak{E}^{1/4}+c\mathfrak{E}^{1/4}\|f_3\|_{L^{2}}^{3/4},\label{tercera}\\
\left\|D^2\mathbf{u}\right\|_{L^2(\partial\Omega_1)}
&\leq c|\lambda|^{1/4}\mathfrak{E}^{3/4}\|\mathbf{w}\|^{1/4}+c\mathfrak{E}^{3/4}\|f_3\|^{1/4}\label{segunda},\\
\left\|D\mathbf{w}\right\|_{L^2(\partial\Omega_1)}
&\leq c|\lambda|^{3/4}\left\| \mathbf{w}\right\|_{L^{2}(\Omega_1)}^{1/4}\left\| \Delta \mathbf{u}\right\|_{L^{2}(\Omega_1)}^{3/4} \nonumber\\ &+ \frac {c}{|\lambda|}\left\| \mathbf{w}\right\|_{L^{2}(\Omega_1)}^{1/4}\left\| \Delta f_1\right\|_{L^{2}(\Omega_1)}^{3/4}\label{primera} ,\\
\left\|\mathbf{w}\right\|_{L^2(\partial\Omega_1)}
&\leq  c|\lambda|^{1/4}\left\|\mathbf{w}\right\|_{L^2(\Omega_1)}^{3/4}\left\|\Delta \mathbf{u}\right\|_{L^{2}(\Omega_1)}^{1/4}+c\left\|\mathbf{w}\right\|_{L^2(\Omega_1)}^{3/4}\left\|\Delta f_1\right\|_{L^{2}(\Omega_1)}^{1/4}\label{cero} .
\end{align}

\end{lemma}
\begin{proof} 
Using Lemma \ref{TrazoInter} and the symmetry of $w$ we get
\begin{eqnarray*}
\int_{\Omega_1} |\nabla\mathbf{w}|^2\;d\Omega
&\leq &c\| \mathbf{w}\|\| \Delta\mathbf{w}\|\nonumber\\
&\leq &c|\lambda|\| \mathbf{w}\|\| \Delta\mathbf{u}\|+c\| \mathbf{w}\|\| \Delta f_1\|.
\label{nabla}
\end{eqnarray*}
Similarly  and using \eqref{f33} we get
 \begin{eqnarray}
\int_{\Omega_1}\left|{\nabla(\kappa_1\Delta  \mathbf{u}}+\beta\vartheta)\right|^2 d \Omega
&\leq &c\left\|{\kappa_1\Delta  \mathbf{u}}+\beta\vartheta\right\|\left\|{\Delta(\kappa_1\Delta  \mathbf{u}}+\beta\vartheta)\right\|\nonumber\\
&\leq &c\left\|{\kappa_1\Delta  \mathbf{u}}+\beta\vartheta\right\|\left\|i\lambda \mathbf{w}-f_3\right\|\nonumber\\
&\leq &c|\lambda|\|\mathbf{w}\|\|\kappa_1\Delta  \mathbf{u}-\beta\vartheta\|+c\|f_3\|\|\kappa_1\Delta  \mathbf{u}-\beta\vartheta\|.\nonumber
\end{eqnarray}
Using Lemma \ref{TrazoInter} for $V=\kappa_1\Delta  \mathbf{u}+\beta\vartheta$ and the symmetry of $V$   we get 
 \begin{eqnarray}
\left\|\nabla(\kappa_1\Delta  \mathbf{u}+\beta\vartheta)\right\|_{L^2(\partial\Omega_1)} 
&\leq &c\|\kappa_1\Delta \mathbf{u} +\beta\vartheta\|_{L^{2}}^{1/4}\|\Delta(\kappa_1\Delta \mathbf{u} +\beta\vartheta)\|_{L^{2}}^{3/4}\nonumber\\
&\leq &c\|\kappa_1\Delta \mathbf{u} +\beta\vartheta\|_{L^{2}}^{1/4}\|\lambda\mathbf{w} -f_3\|_{L^{2}}^{3/4}\nonumber\\
&\leq &c|\lambda|^{3/4}\|\mathbf{w}\|^{3/4}c\|\kappa_1\Delta \mathbf{u} +\beta\vartheta\|_{L^{2}}^{1/4}+c\|\kappa_1\Delta \mathbf{u} +\beta\vartheta\|_{L^{2}}^{1/4}\|f_3\|_{L^{2}}^{3/4}\nonumber\\
&\leq &c|\lambda|^{3/4}\|\mathbf{w}\|^{3/4}\mathfrak{E}^{1/2}+c\mathfrak{E}^{1/4}\|f_3\|_{L^{2}}^{3/4}.\nonumber\
\end{eqnarray}
By the symmetry of $\mathbf{u}$ we have 
$\|D^2 \mathbf{u}\|_{L^2(\partial\Omega_1)}\leq c\|\kappa_1\Delta \mathbf{u} +\beta\vartheta\|_{L^2(\partial\Omega_1)}$ and
 \begin{eqnarray*}
\left\|D^2 \mathbf{u}\right\|_{L^2(\partial\Omega_1)}
&\leq&  c|\lambda|^{1/4}\mathfrak{E}^{3/4}\|\mathbf{w}\|^{1/4}+c\mathfrak{E}^{3/4}\|f_3\|^{1/4}.\nonumber 
\end{eqnarray*}
Similarly
 \begin{eqnarray*}
\left\|\nabla\mathbf{u}\right\|_{L^2(\partial\Omega_1)}
&\leq&  \frac {c}{|\lambda|^{1/4}}\mathfrak{E}+ \frac {c}{|\lambda|}\left\| \mathbf{w}\right\|_{L^{2}(\Omega_1)}^{1/4}\left\| \Delta f_1\right\|_{L^{2}(\Omega_1)}^{3/4}.\nonumber
\end{eqnarray*}
And finally 
 \begin{eqnarray*}
\left\|\mathbf{w}\right\|_{L^2(\partial\Omega_1)}
&\leq&  c|\lambda|^{1/4}\left\|\mathbf{w}\right\|_{L^2(\Omega_1)}^{3/4}\left\|\Delta \mathbf{u}\right\|_{L^{2}(\Omega_1)}^{1/4}+c\left\|\mathbf{w}\right\|_{L^2(\Omega_1)}^{3/4}\left\|\Delta f_1\right\|_{L^{2}(\Omega_1)}^{1/4}\nonumber,
\end{eqnarray*}

for $\lambda$ large. 

\end{proof}

\begin{lemma} \label{utheta}
Over the thermoelastic part, we get
\begin{eqnarray*}
	\int_{\Omega_1}\left| \Delta  \mathbf{u}\right|^2 d\Omega&\leq&
	\frac{c}{|\lambda|^{1/2}}\mathfrak{R}_0\mathfrak{E}+\frac{c}{|\lambda|^{1/4}}\mathfrak{R}_0\mathfrak{E}.\label{Delta}
\end{eqnarray*}
\end{lemma}	 
\begin{proof}
Multiplying the equation \eqref{eta} by $\overline{\Delta  \mathbf{u}}$ and integration on $\Omega_1$   we see
\begin{align}
\int_{\Omega_1}\dfrac{\beta}{a_n}\left|\Delta   \mathbf{u}\right| ^2 \ d\Omega=  \underbrace{\int_{\Omega_1}\Theta_n\overline{\Delta  \mathbf{u}} \ d\Omega-\dfrac{1}{a_n i\lambda}\int_{\Omega_1}\sum_{j=0}^{n-1}a_j\Theta_{j}\overline{\Delta \mathbf{u}} \ d\Omega-\dfrac{1}{i\lambda }\int_{\Omega_1}g_{n+1}\overline{\Delta  \mathbf{u}} \ d\Omega}_{\leq c\|\Theta_n\|^2+\epsilon\|\Delta \mathbf{u}\|^2+\frac{c}{|\lambda|^2}\mathfrak{R}^2} \nonumber\\-\underbrace{\dfrac{1}{a_n i\lambda}\int_{\Omega_1}\Delta f_1\overline{\Delta  \mathbf{u}} \ d\Omega}_{\leq \epsilon\|\Delta \mathbf{u}\|^2+\frac{1}{|\lambda|^2}\|F\|_{\mathcal{H}}^2}-\underbrace{\dfrac{1}{a_n i \lambda}\int_{\Omega_1}\sum_{j=0}^{n}b_j\Delta \Theta_j\overline{\Delta  \mathbf{u}} \ d\Omega}_{I_2}+\dfrac{\mu}{i\lambda}\int_{\Omega_1}\Theta_n\overline{\Delta \mathbf{u}}\ d \Omega.\label{deltau2}
\end{align}
Using Green's Formula and Lemma \ref{etanTT} it follows that
\begin{eqnarray}
	\left|I_2 \right| &=&\left| \dfrac{1}{a_n i \lambda}\int_{\Omega_1}\sum_{j=0}^{n}b_j\Theta_j\overline{\Delta^2  \mathbf{u}} \ d\Omega+\dfrac{1}{a_n i \lambda}\int_{\Gamma_0}\sum_{j=0}^{n}b_j\frac{\partial \Theta_j}{\partial\nu}\overline{\Delta \mathbf{u}} \ d\Gamma\right|\nonumber\\
	&\leq &c\|\Theta\|\|\mathbf{w}\|+\frac{c}{|\lambda|^{1/4}}\mathfrak{R}_0\mathfrak{E}. \label{i22}
\end{eqnarray}

Finally, replacing  \eqref{i22} into \eqref{deltau2} and taking the real part we get
\begin{align*}
	\int_{\Omega_1} \left| \Delta \mathbf{u}\right| ^2 d\Omega\leq 	\frac{c}{|\lambda|^{1/4}}\mathfrak{R}_0\mathfrak{E} 	 +c\|\Theta_n\|\|\mathbf{w}\|+
	\dfrac{c}{|\lambda|}\|U\|_{\mathcal{H}}\|F\|_{\mathcal{H}}. 
\end{align*}
So our conclusion follows. 
\end{proof}

\begin{lemma}\label{lew}
	Over the thermoelastic part, we have
	\begin{align*}
		\int_{\Omega_1}|\mathbf{w}|^2+\kappa_1 |\Delta\mathbf{u}|^2+ \left|\vartheta\right|^2   \ d\Omega\leq \frac{c_\epsilon}{|\lambda|}\|U\|_{\mathcal{H}}\|F\|_{\mathcal{H}} +\frac{c_\epsilon}{|\lambda|^2}\|F\|_{\mathcal{H}}^2.
	\end{align*}
	
\end{lemma}
\begin{proof}
	We multiply \eqref{f33} by $\frac{1}{i\lambda}\overline{\mathbf{w}}$ and using \eqref{f11} 
	\begin{align*}
		\int_{\Omega_1}	 |\mathbf{w}	|^2  d\Omega+\frac{1}{i\lambda}\int_{\Omega_1} \mu |\mathbf{w}|^2d\Omega+\frac{1}{i\lambda}\int_{\Omega_1}\Delta(\kappa_1\Delta \mathbf{u}+\beta\vartheta)\overline{ \mathbf{w}} \ d \Omega=\frac{1}{i\lambda}\int_{\Omega_1} f_3\overline{\mathbf{w}} \ d \Omega.
	\end{align*}
	Using Green theorem 
	\begin{align}
		\int_{\Omega_1} |\mathbf{w}|^2 d\Omega&=\underbrace{\frac{1}{i\lambda}\int_{\Gamma_0}\left( \kappa_1\frac{\partial \Delta  \mathbf{u}}{\partial\nu} +\beta\frac{\partial\vartheta}{\partial\nu}\right) \overline{ \mathbf{w}} \ d \Gamma}_{:=I_1} -\underbrace{\int_{\Gamma_0}\frac{1}{i\lambda}\left( \kappa_1\Delta  \mathbf{u} +\beta\vartheta\right) \overline{ \frac{\partial\mathbf{w}}{\partial\nu}} \ d \Gamma}_{:=I_2} 
		-\frac{1}{i\lambda}\int_{\Omega_1}\beta \vartheta \overline{\Delta \mathbf{w}} d\Omega\nonumber\\
		&+\int_{\Omega_1}\kappa_1|\Delta \mathbf{u}|^2d\Omega-\frac{1}{i\lambda}\int_{\Omega_1} f_3\overline{\mathbf{w}} \ d \Omega -\frac{1}{i\lambda}\int_{\Omega_1} \mu  |\mathbf{w}|^2d\Omega. \label{acw}
	\end{align}
	From inequalities \eqref{tercera} and \eqref{cero} we get
	\begin{eqnarray*}
		\left|I_1\right|
		&\leq&\frac{1}{|\lambda|}\int_{\Gamma_0}\left| \kappa_1\frac{\partial \Delta  \mathbf{u}}{\partial\nu}+\beta\frac{\partial\vartheta}{\partial\nu}\right| |\overline{ \mathbf{w}}| \ d \Gamma\\
		&\leq &\frac{c}{|\lambda|}\left(|\lambda|^{3/4}\|\mathbf{w}\|^{3/4}\mathfrak{E}^{1/4}+\mathfrak{E}^{1/4}\|f_3\|_{L^{2}}^{3/4}\right)
		\left( c|\lambda|^{1/4}\|\mathbf{w}\|^{3/4}\|\Delta \mathbf{u}\|^{1/4}+c\|\mathbf{w}\|^{3/4}\|\Delta f_3\|^{1/4}\right) \\
		&\leq & c\mathfrak{E}^{1/4}\|\mathbf{w}\|^{3/2}\|\Delta \mathbf{u}\|^{1/4}+\frac{c}{|\lambda|^{3/4}}\mathfrak{E}^{1/4}\|f_3\|^{3/4}\|\mathbf{w}\|^{3/4}\|\Delta \mathbf{u}\|^{1/4}\\
		&&+\frac{c}{|\lambda|^{1/4}}\mathfrak{E}^{1/4}\|\mathbf{w}\|^{3/2}\|\Delta f_3\|^{1/4}+\frac{c}{|\lambda|}\mathfrak{E}^{1/4}\|f_3\|_{L^{2}}^{3/4}\|\mathbf{w}\|^{3/4}\|\Delta f_3\|^{1/4}\\
		&\leq & c \mathfrak{E}^2+\frac 12 \|\mathbf{w}\|^{2}+\frac 1{2c} \|\Delta \mathbf{u}\|^{2}+\frac{c}{|\lambda|^2}\|F\|_{\mathcal{H}}^2.
	\end{eqnarray*}
	Similarly,  
	\begin{eqnarray*}
		\left|I_2\right|
		&\leq&\| \Delta  \mathbf{u}\|_{L^{2} (\partial\Omega)}\|\frac{\partial  \mathbf{w}}{\partial\nu}\|_{L^{2} (\partial\Omega)}\\
		&\leq&\frac{c}{|\lambda|^{1/4}}\left(c|\lambda|^{1/4}\mathfrak{E}^{1/4}\|\mathbf{w}\|^{1/4}\left\|\Delta \mathbf{u}\right\|_{L^2(\Omega_1)}^{1/2}+c\mathfrak{E}^{3/4}\|f_3\|^{1/4}\right)\left(\left\| \mathbf{w}\right\|_{L^{2}(\Omega_1)}^{1/4}\left\| \Delta\mathbf{u}\right\|_{L^{2}(\Omega_1)}^{3/4}\right).
	\end{eqnarray*}
	Using similar arguments we get 
	\begin{eqnarray*}
		\left|I_2\right|
		&\leq&  c \mathfrak{E}^2+\frac 12 \|\mathbf{w}\|^{2}+\frac 1{2c} \|\Delta \mathbf{u}\|^{2}+\frac{c}{|\lambda|^2}\mathfrak{R}^2.
	\end{eqnarray*}
	Inserting the above inequalities into \eqref{acw} we find
	
	\begin{eqnarray*}
		\int_{\Omega_1}	 |\mathbf{w}|^2 d\Omega&\leq&c \mathfrak{E}^2+\frac 1{2c} \|\Delta \mathbf{u}\|^{2}+\frac{c}{|\lambda|^2}\mathfrak{R}^2 .
	\end{eqnarray*}
Now,	using Lemma \ref{uthetapp}  and Lemma \ref{utheta} our conclusion follows. 
	
\end{proof}

\begin{lemma}\label{estff} Under the above conditions we have 
\begin{eqnarray}\label{fz1}
	\left|\int_{\Omega _2}  \mathbf{z}\mathbf{q}.\overline{ \nabla f_2}  d\Omega \right| &\leq &
	\dfrac{c}{\left|\lambda \right|^{1/2} } \left\|F \right\|_{\mathcal{H}}^2+\epsilon \left\|U \right\|_{\mathcal{H}},\\
		\left| \int_{\Omega _2} f_4\mathbf{q}.\overline{ \nabla \mathbf{v}} \ d\Omega\right| &\leq &
		 \dfrac{c}{\left| \lambda\right| } \left\|F \right\|_{\mathcal{H}}^2+\epsilon \left\|U \right\| \label{fz2}^2_{\mathcal{H}} .
\end{eqnarray}
\end{lemma}
\begin{proof}
Using the same procedure as in \eqref{estxxx} we get
\begin{align}\label{zv1}
\left\|\nabla \mathbf{z} \right\|\leq \left\| \mathbf{z}\right\|^{1/2}\left\|\Delta \mathbf{z} \right\|^{1/2} \leq c  \left\| \mathbf{z}\right\|^{1/2}\left\| i\lambda \Delta \mathbf{v}+\mu \Delta \mathbf{v}+\Delta f_2 \right\|^{1/2} \nonumber \\
\leq c\left|\lambda \right|^{1/2}\|U \| _{\mathcal{H}}+c\mathfrak{R}^2.
\end{align}
Using \eqref{zv1} and taking the real part  we get
\begin{align*}
\left|	\int_{\Omega _2}  \mathbf{q} f_4\overline{\nabla \mathbf{v}} \ d\Omega\right|&=\left|\dfrac{1}{i\lambda}\int_{\Omega _2}  f_4\mathbf{q}.\left( \nabla \mathbf{z}+\nabla f_2-\mu \nabla \mathbf{v}\right)  \ d\Omega\right|
	\leq \dfrac{c}{\left| \lambda\right|^{1/2} }\|U \| _{\mathcal{H}}\|F \| _{\mathcal{H}}+\dfrac{c}{| \lambda| }\|F \| _{\mathcal{H}}^2,
\end{align*}
for $\lambda$ large. So, we use \eqref{f44} to find
\begin{align*}
\left|	\int_{\Omega _2}  \mathbf{q} f_4\overline{\nabla \mathbf{v}} \ d\Omega\right|&\leq \dfrac{c}{\left| \lambda\right| } \left\|F \right\|_{\mathcal{H}}^2+\epsilon \left\|U \right\|^2_{\mathcal{H}} .
\end{align*}
So we get \eqref{fz1}. Finally, using \eqref{f22} and taking the real part
\begin{align*}
\inte{}{\Omega_2} \mathbf{z}\mathbf{q}.\overline{\nabla f_2} \ d\Omega &=\dfrac{1}{i\lambda}\inte{}{\Omega_2} i \lambda   \mathbf{z} \mathbf{q}.\overline{\nabla f_2} \ d\Omega 
	= \dfrac{1}{i\lambda}\inte{}{\Omega_2} (f_4-\mu \mathbf{z}-\kappa_2 \Delta^2 \mathbf{v}) \mathbf{q}.\overline{\nabla f_2} \ d\Omega \\
	&=\underbrace{\dfrac{1}{i\lambda}\inte{}{\Omega_2} (f_4-\mu \mathbf{z}) \mathbf{q}\cdot\overline{\nabla f_2} \ d\Omega}_{\leq \frac{c}{|\lambda|}\mathfrak{R}}-\underbrace{\dfrac{1}{i\lambda}\int_{\Gamma_0}\kappa_2\dfrac{\partial \Delta \mathbf{v}}{\partial \nu}\mathbf{q}.\overline{\nabla f_2} \ d\Omega}_{=I_2}\\
	& + \underbrace{\dfrac{1}{i\lambda}\int_{\Omega_2} \kappa_2\nabla \Delta \mathbf{v}\cdot \nabla (\mathbf{q}.\overline{\nabla f_2} )\ d\Omega}_{=I_3}
\end{align*}
Using the transmission conditions \eqref{tercera} we have
\begin{align*}
\left|I_2 \right|\leq \dfrac{c}{\left| \lambda\right|^{1/4}  }\|U\|_{\mathcal{H}}\|F\|_{\mathcal{H}}+\dfrac{c}{\left| \lambda\right|^{5/4}  }\|F\|_{\mathcal{H}}.
\end{align*}
Finally, using the symmetry we get
$$
\left|I_3\right|\leq  \frac{c}{| \lambda|} \|\nabla \Delta \mathbf{v}\|\|F\|\leq \frac{c}{| \lambda|} \|\Delta v\|^{1/2}\|\Delta^2 v\|^{1/2}\|F\|\leq \frac{c}{| \lambda|} \|\Delta v\|^{1/2}\|i\lambda z+f_4\|^{1/2}\|F\|_{\mathcal{H}}.
$$
So we have 
$$
\left|I_3\right|\leq   \frac{c}{| \lambda|^{1/2}} 
\|\Delta v\|^{1/2}\| z\|^{1/2}\|F\|+\frac{c}{| \lambda|} \|\Delta v\|^{1/2}\| f_4\|^{1/2}\|F\|_{\mathcal{H}}.
$$
Finally we arrive to
\begin{align*}
		\left|\inte{}{\Omega_2} \mathbf{z}\mathbf{q}.\overline{\nabla f_2}  \ d\Omega \right| \leq \epsilon \|U\|_{\mathcal{H}}+\dfrac{c}{\left| \lambda\right| }\|F\|_{\mathcal{H}},
\end{align*}
for $\lambda$ large. Our conclusion follows. 
\end{proof}

\begin{lemma} \label{lem3.2} Let us denote by $q_k$ a first order polynomial.  
	Under the above notations we have 
	\begin{eqnarray*}
		\int_{\Omega}\Delta^2\varphi\cdot \left(q_k \frac{\partial\varphi}{\partial x_k}\right) d\Omega
		&=&\int_{\partial\Omega}\frac{\partial\Delta\varphi}{\partial\nu}
\cdot\left( q_k\frac{\partial\varphi} {\partial x_k}\right) d\Gamma 
		+\frac 12 \int_{\partial\Omega}\mathbf{q}\cdot\nu |\Delta\varphi|^2d\Gamma 
		+\int_{\partial\Omega}\Delta\varphi \frac{\partial\varphi}{\partial\nu} d\Gamma\\
		&&
		-\int_{\partial\Omega} (\Delta\varphi)\cdot  q_k\left(\frac{\partial\nabla\varphi} {\partial x_k}\cdot\nu\right) d\Gamma+\int_{\Omega} |\Delta\varphi |^2 d\Omega. \\
	\end{eqnarray*}
	
\end{lemma}
\newcommand{\dis}{\displaystyle}
\begin{proof} 
Using integration by parts we have 
$$
\begin{array}{r c l}
&& \dis\int_{\Omega}\Delta^2\varphi\cdot \left( q_k\frac{\partial\varphi} {\partial x_k}\right) d\Omega =
\dis \int_{\partial\Omega}\frac{\partial\Delta\varphi}{\partial\nu}
\cdot\left( q_k\frac{\partial\varphi} {\partial x_k}\right) d\Gamma\underbrace{-\int_{\Omega}\nabla(\Delta\varphi) \nabla\cdot\left( q_k\frac{\partial\varphi} {\partial x_k}\right) d\Omega}_{I}.
\end{array}
$$
On the other hand 

\begin{eqnarray*} 
 I &=&
  -\int_{\Omega}\nabla(\Delta\varphi) \cdot\left(\nabla \varphi+q_k\frac{\partial\nabla \varphi} {\partial x_k}\right) d\Omega\\
& = & \dis-\int_{\partial\Omega}\Delta\varphi \frac{\partial\varphi}{\partial\nu} d\Gamma
+\int_{\Omega} |\Delta\varphi |^2 d\Omega -\int_{\partial\Omega} (\Delta\varphi)\cdot  q_k\left(\frac{\partial\nabla\varphi} {\partial x_k}\cdot\nu\right) d\Gamma\\
&&\dis +\int_{\Omega}\Delta\varphi (\nabla q_k)\cdot\left(\frac{\partial\nabla\varphi} {\partial x_k}\right) d\Omega
+\int_{\Omega}\Delta\varphi q_k\left(\frac{\partial\Delta\varphi} {\partial x_k}\right) d\Omega.
\end{eqnarray*}

\noindent Note  that 
\begin{align*}
\int_{\Omega}\Delta\varphi (\nabla q_k)\cdot\left(\frac{\partial\nabla\varphi} {\partial x_k}\right) d\Omega&=\int_{\Omega}|\Delta\varphi |^2 d\Omega .\\
\int_{\Omega}\Delta\varphi q_k\left(\frac{\partial\Delta\varphi} {\partial x_k}\right) d\Omega&=\frac 12 \int_{\Omega}q_k\frac{\partial|\Delta\varphi|^2} {\partial x_k} d\Omega.
\end{align*}

From where our conclusion follows.
\end{proof}

\begin{lemma} \label{elas}
For the problem over $\Omega_2$ we have that the following estimate
	\begin{align*}
	\int_{\Omega_2 } \left|\mathbf{z} \right| ^2+ \kappa_2\left| \Delta \mathbf{v} \right| ^2 \ d\Omega \leq \epsilon \left\|U \right\|_{\mathcal{H}}^2+\dfrac{c\left\|F \right\|^2_{\mathcal{H}} }{\left|\lambda \right|^{1/2} },
	\end{align*}
is satisfied.
\end{lemma}
\begin{proof}
	Multiplying \eqref{f44} by $q_k\overline{\dfrac{\partial \mathbf{v}}{\partial x_k}}$, with $\mathbf{q}=(x_1,x_2)$ and using Lemma \ref{lem3.2} we have
		\begin{eqnarray*}
	 \int_{\Omega_2 }  |\mathbf{z}|^2+\kappa_2 |\Delta \mathbf{v}|^2\ d\Omega&=&\underbrace{	\int_{\Omega_2 } f_4 q_k\overline{\frac{\partial \mathbf{v}}{\partial x_k}} \ d\Omega+ \int_{\Omega_2 }\mathbf{z}q_k\overline{\frac{\partial f_2}{\partial x_k}} d\Omega}_{=I_3}
	 +\int_{\partial \Omega_2 }  q_k\nu_k |\mathbf{z}|^2 d\Gamma\\
	 &&+\underbrace{\int_{\partial\Omega_2}\frac{\partial\Delta \mathbf{v}}{\partial\nu}
\cdot\left( q_k\frac{\partial \mathbf{v}} {\partial x_k}\right) d\Gamma }_{\leq c\|\nabla^3 \mathbf{v}\|_{L^2(\partial\Omega)}\|\nabla\mathbf{v}\|_{L^2(\partial\Omega)}}
		+ \underbrace{\frac 12\int_{\partial\Omega_2}q\cdot\nu |\Delta \mathbf{v}|^2d\Gamma }_{\leq c\|\nabla^2 \mathbf{v}\|_{L^2(\partial\Omega)}^2}\\
		&&+\underbrace{\int_{\partial\Omega_2}\Delta \mathbf{v} \frac{\partial \mathbf{v}}{\partial\nu} d\Gamma }_{\leq c\|\nabla^2 \mathbf{v}\|_{L^2(\partial\Omega)}\|\nabla \mathbf{v}\|_{L^2(\partial\Omega)}}		-\underbrace{\int_{\partial\Omega_2} (\Delta \mathbf{v})\cdot  q_k\left(\frac{\partial\nabla \mathbf{v}} {\partial x_k}\cdot\nu\right) d\Gamma }_{\leq c\|\nabla^2 \mathbf{v}\|_{L^2(\partial\Omega)}^2}.
	\end{eqnarray*}
\noindent
Note that the volume integral in the equation above are bounded because of Lemma \ref{estff}
$$
\left|I_3\right|\leq 
\dfrac{c}{\left|\lambda \right|^{1/2} } \left\|F \right\|_{\mathcal{H}}^2+\epsilon \left\|U \right\|_{\mathcal{H}}.
$$
Using the transmission conditions  and the symmetry Lemma \ref{TrazoInter}
\begin{align*}
\|\nabla^3 \mathbf{v}\|_{L^2(\partial\Omega)}\|\nabla \mathbf{v}\|_{L^2(\partial\Omega)}\leq 
\dfrac{c}{\left|\lambda \right|^{3/8} } \left\|F \right\|_{\mathcal{H}}^2+\epsilon \left\|U \right\|_{\mathcal{H}}.\\
\|\nabla^2 \mathbf{v}\|_{L^2(\partial\Omega)}^2
\leq \dfrac{c}{\left|\lambda \right|^{3/8} } \left\|F \right\|_{\mathcal{H}}^2+\epsilon \left\|U \right\|_{\mathcal{H}}.\\
\left| \int_{\partial\Omega_2}\Delta \mathbf{v} \frac{\partial \mathbf{v}}{\partial\nu} d\Gamma\right| 
+
\left| \int_{\partial\Omega_2} (\Delta \mathbf{v})\cdot  q_k\left(\frac{\partial\nabla \mathbf{v}} {\partial x_k}\cdot\nu\right) d\Gamma\right| \leq \dfrac{c}{\left|\lambda \right|^{3/8} } \left\|F \right\|_{\mathcal{H}}^2+\epsilon \left\|U \right\|_{\mathcal{H}}.
\end{align*}
Therefore,
\begin{align*}
	\int_{\Omega _2}\kappa_2 \left|\Delta \mathbf{v} \right| ^2+\rho_2 \left| \mathbf{z}\right| ^2 \ d\Omega\leq \epsilon \left\|U \right\|_{\mathcal{H}}^2+\dfrac{c\left\|F \right\|^2_{\mathcal{H}} }{\left|\lambda \right|^{1/2} }.
\end{align*}
\end{proof}
\begin{theorem}\label{Maintheorem}
	The phase-lag thermoelastic system \eqref{Lb1}-\eqref{Lb33xx} is of $4$-Gevrey's class to radial solutions. 
\end{theorem}
\begin{proof}
	Using Lemma \ref{elas} and Lemma \ref{lew} over the elastic component and thermoelastic component
	\begin{align}
		\int_{\Omega_1}|\mathbf{w}|^2+\kappa_1 |\Delta \mathbf{u}|^2+ c\left|\vartheta\right|^2   \ d\Omega&\leq \frac{c_\epsilon}{|\lambda|}\|U\|_{\mathcal{H}}\|F\|_{\mathcal{H}} +\frac{c_\epsilon}{|\lambda|^2}\|F\|_{\mathcal{H}}^2\label{dos}\\
	\int_{\Omega_2 } \left|\mathbf{z}\right| ^2+ \kappa_2\left| \Delta \mathbf{v} \right| ^2 \ d\Omega &\leq \epsilon \left\|U \right\|_{\mathcal{H}}^2+\dfrac{c\left\|F \right\|^2_{\mathcal{H}} }{\left|\lambda \right|^{1/2} }.\label{tres}
	\end{align}
	From  \eqref{dos}, \eqref{tres} we get 
	\begin{align*}
		\|U\|_{\mathcal{H}}^{2} \leq  c\epsilon\|U\|_{\mathcal{H}}^{2}+\frac c{|\lambda|^{1/2}}\|F \|_\mathcal{H}^2.
	\end{align*}
	Taking $\epsilon$ small our conclusion follows. 
\end{proof}

\section*{Acknowledgments} 

Jaime Muñoz Rivera was supported by CNPq project 307947/2022-0  and Fondecyt project 1230914.

\section*{Conflict of interest}	
This work does not have any conflicts of interest.

\end{document}